\documentclass[a4paper]{article}
\usepackage{a4wide}
\usepackage{authblk}
\makeatletter
\providecommand*{\shuffle}{%
  \mathbin{\mathpalette\shuffle@{}}%
}
\newcommand*{\shuffle@}[2]{%
  \sbox0{$#1\vcenter{}$}%
  \kern .15\ht0 
  \rlap{\vrule height .25\ht0 depth 0pt width 2.5\ht0}%
  \raise.1\ht0\hbox to 2.5\ht0{%
    \vrule height 1.75\ht0 depth -.1\ht0 width .17\ht0 %
    \hfill
    \vrule height 1.75\ht0 depth -.1\ht0 width .17\ht0 %
    \hfill
    \vrule height 1.75\ht0 depth -.1\ht0 width .17\ht0 %
  }%
  \kern .15\ht0 
}
\makeatother

\usepackage{graphicx}
\usepackage{tikz}
\usetikzlibrary{fit,patterns,decorations.pathmorphing,decorations.pathreplacing,calc,arrows}

\usepackage{amsmath,amssymb,amsthm}
\usepackage[inline]{enumitem}

\usepackage{thmtools}
\usepackage{hyperref}
\usepackage[capitalize,nameinlink,sort]{cleveref}

\usepackage{mleftright}
\mleftright

\usepackage[T1]{fontenc}
\usepackage{newpxtext}
\usepackage{euler}

\usepackage[textsize=scriptsize]{todonotes}

\DeclareMathOperator{\inv}{inv}
\DeclareMathOperator{\ind}{ind}
\DeclareMathOperator{\per}{per}
\DeclareMathOperator{\Fac}{Fac}
\DeclareMathOperator{\val}{val}
\DeclareMathOperator{\N}{\mathbb{N}}

\newcommand{\qbin}[2]{\binom{#1}{#2}_{\!\! q}}

\declaretheorem[numberwithin=section]{theorem}
\declaretheorem[sibling=theorem]{lemma,corollary,proposition}
\declaretheorem[sibling=theorem,style=definition]{example,definition,remark}
\declaretheorem[sibling=theorem,style=definition,refname={fact,facts},Refname={Fact,Facts}]{fact}

\declaretheoremstyle[
    headfont=\normalfont\itshape, 
    bodyfont = \normalfont,
    qed=$\blacksquare$, 
    headpunct={:}]{claimproofstyle} 

\crefname{equation}{}{}

\title{Introducing $q$-deformed binomial coefficients of words}

\author{Antoine Renard}

\author{Michel Rigo\thanks{Supported by the FNRS Research grant T.196.23 (PDR), ORCID 0000-0001-7463-8507}}

\author{Markus A. Whiteland\thanks{Supported by the FNRS Research grant 1.B.466.21F, ORCID 0000-0002-6006-9902.}}

\affil{Department of Mathematics, University of Li\`ege, Li\`ege, Belgium}

\affil{\texttt{\{antoine.renard,m.rigo,mwhiteland\}@uliege.be}}

\date{}

\begin{document}
\maketitle            
\begin{abstract}
Gaussian binomial coefficients are $q$-analogues of the binomial coefficients of integers. On the other hand,
binomial coefficients have been extended to finite words, i.e., elements of a finitely generated free monoid. In this
paper we bring these two notions together by introducing $q$-analogues of binomial coefficients of words. We
study their basic properties, e.g., by extending classical formulas such as the $q$-Vandermonde and Manvel--Meyerowitz--Schwenk--Smith--Stockmeyer identities to our setting. These $q$-deformations contain much richer
information than the original coefficients. From an algebraic perspective, we introduce a $q$-shuffle
and a family of $q$-infiltration products for non-commutative formal power series. Finally, we apply our results to
generalize a theorem of Eilenberg characterizing so-called $p$-group languages. We show that a language is of this type if and only if it is a Boolean combination of specific languages defined through $q$-binomial coefficients seen as polynomials over $\mathbb{F}_p$.
\end{abstract}



\section{Introduction}

When defining what is called a {\em $q$-analogue} of a counting function, one gets a polynomial in the variable $q$ with non-negative integer coefficients which reduces to the original function when $q$ tends to $1$. To be a useful deformation, one seeks to satisfy variations or
adaptations of some, if not all, of the algebraic properties of the original function. For instance, $q$-deformations are useful in constructing
generating functions or in obtaining new combinatorial identities. They find applications in physics and mathematics, in particular in number
theory (e.g.,~Hurwitz polyzeta functions \cite{Duchamp2017}). To get a glimpse of the $q$-mathematics, let us also mention quantum algebras that
are $q$-deformations of Lie algebras with important applications in mathematical physics. See, for instance, \cite{Andrews}. In particular, we were inspired by the work of Morier-Genoud and Ovsienko on $q$-deformed rationals, continued fractions, and reals \cite{MRO,MRO2}. Their recursive formula is analogous to the $q$-deformed Pascal identity for the Gaussian binomial coefficients, but they replace the Pascal triangle with the Farey graph. Building on this work $q$-rational and $q$-real binomial coefficients were considered in \cite{machacek2023qrational}.

Binomial coefficients of finite words have a central position in combinatorics on words  \cite{Dudik,LejeuneRosenfeld,ManvelMSSS1991reconstruction,RigoSalimov} and formal language theory. A celebrated theorem of Eilenberg, also credited to Sch\"utzenberger, \cite[Thm.~VIII.10.1]{Eilenberg1976} provides a characterization of $p$-group languages using these binomial coefficients. One of the main applications of our work is to obtain a generalization of this result working with the $q$-analogues of these coefficients. This opens the way to new perspectives in the study of these languages and of the groups themselves, with numerous applications \cite{Pin1,Pin2,ReutenauerFree}. For instance, relying on a systematic use of the binomial coefficients, the authors of \cite{Pin3} describe the two classes of languages recognized by the dihedral group $D_4$ and the quaternion group $Q_8$.

 \subsection{Gaussian binomial coefficients}
The Gaussian binomial coefficients are $q$-analogues of the binomial coefficients. They are polynomials in the variable~$q$ defined as
\[
\qbin{m}{r}
= 
\frac{(1-q^m)(1-q^{m-1})\cdots(1-q^{m-r+1})} {(1-q)(1-q^2)\cdots(1-q^r)},
\]
where $m\ge r\ge 0$ are integers. There are several combinatorial properties associated to Gaussian binomial coefficients. For example,
the coefficient of the monomial $q^i$ is the number of partitions of $i$ whose Ferrers diagram\footnote{A finite collection of unit squares called cells, arranged in an array of left-justified rows, with the row lengths in non-increasing order from top to bottom.} (or, Young tableau) fits in a $r \times (m-r)$ box. From the above definition, one gets for instance MacMahon's $q$-Catalan numbers that give the major index generating function on Dyck paths \cite{Haglund}.

\subsection{Binomial coefficients of words} 
We refer the reader to \cite{Lothaire1997} for a classical textbook on combinatorics on words. An alphabet $A$ is a finite set and a word is a finite sequence of elements in $A$. The free monoid $A^*$ is the set of words over $A$ equipped with the concatenation product and the empty word $\varepsilon$ is the identity element. We let $|x|$ denote the length of the word $x$.
The (classical) binomial coefficient of two words $u$ and $v$ counts how many times $v$ occurs as a subword of $u=a_1a_2\cdots a_n$, $a_i\in A$:
\[
\binom{u}{v}=\#\left\{i_1<\cdots <i_{|v|} \mid a_{i_1}a_{i_2}\cdots a_{i_{|v|}}=v \right\}.
\]
Its $q$-analogue is the main object of study of this paper. As we will see, this polynomial provides extra information about where those subwords occur. As a trivial example when $v$ is reduced to a single letter, the binomial coefficient simply counts the number of occurrences of this letter while its $q$-analogue is a polynomial where the monomials give the positions of these occurrences (starting from $0$ and counted from the right). For example,
\[
\binom{01001201021}{0}=5 \quad\text{ and }\quad\qbin{01001201021}{0}=q^{10}+q^8+q^7+q^4+q^2.
\]
As the reader may notice, evaluating the latter polynomial at $q=1$ gives back the value of the classical binomial coefficient of words. 

The fact that $q$-binomial coefficients contain more information than merely counting subwords is exemplified when considering the famous problem
of reconstructing a word from some of its binomial coefficients (see, for instance, \cite{ManvelMSSS1991reconstruction}). As an example, the two
words $u=0110$ and $v=1001$ are said to be $2$-binomially equivalent \cite{RigoSalimov}, i.e., $\binom{u}{x}=\binom{v}{x}$ for all words $x$ of
length at most $2$. Hence, the sole knowledge of these coefficients is not enough to uniquely determine the word. With their $q$-deformation
counterparts, the knowledge of $\qbin{u}{0}=q^3+1$ and $\qbin{u}{1}=q^2+q$ completely determines the word~$u$. In general, a word $u\in A^*$ is uniquely determined by the polynomials $(\qbin{u}{a})_{a\in A}$, see \cref{cor:reconstruction}.

In this paper, we define $q$-analogues of binomial coefficients of words \cref{eq:recdef} based on an analogue of a Pascal's identity satisfied by Gaussian binomial coefficients. We will see that most of the properties and identities of binomial coefficients of words can be adapted. This also holds for a $q$-analogue of the shuffle of two words which is a formal polynomial. 

When binomial coefficients of words are encountered, it is quite usual in the classical theory to study Parikh matrices because they contain such coefficients. In a series of papers first initiated by E{\u{g}}ecio{\u{g}}lu \cite{Eeciolu2004,Egecioglu}, $q$-analogues of Parikh matrices (whose upper diagonal entries are binomial coefficients of words) have been considered. Indeed, the $q$-counting found in \cite{Eeciolu2004} is the same as ours (up to a convenient power of $q$). However, ours is more natural due to the many links with Gaussian binomial coefficients established here.  
The polynomials that appear in \cite{Egecioglu} as coefficients of these matrices are related to a given subword $v=a_1\cdots a_k$ where $a_i\in A$. Each factorization of $u=u_1a_1u_2\cdots u_ka_ku_{k+1}$ where $u_i\in A^*$ provides the polynomial related to $v$ with a monomial $\prod_{i=1}^kq^{|u_i|_{a_i}}$ where $|u_i|_{a_i}$ is the number of letters $a_i$ occurring in $u_i$. This clearly cannot be compared with our contribution. 

\subsection{Our contributions}
This paper is organized as follows. In \cref{sec:2}, we define $q$-binomial coefficients of words. In \cref{sec:combinter}, we provide a combinatorial interpretation of these $q$-deformations. With \cref{thm:powers} we determine how each specific occurrence of $v$ as a subword of $u$ contributes to $\qbin{u}{v}$. This result has several important corollaries. In particular, we obtain a $q$-Vandermonde identity expressing $\qbin{xy}{u}$ as a sum of terms involving products of the form $\qbin{x}{u_1}\qbin{y}{u_2}$ where $u_1u_2$ is a factorization of $u$. In \cref{sec:5} we reconsider several classical identities in the context of $q$-binomial coefficients of words.

We then study algebraic properties of $q$-deformations in the next two sections. In \cref{sec:def}, we focus on studying $q$-deformations of the shuffle and infiltration products. Both these notions speak about formal polynomials, i.e., applications from $A^*\times A^*$ to
$\mathbb{N}[q]\langle A^*\rangle$. 
 In the literature, $q$-shuffle algebras were independently studied by Rosso \cite{RossoCRAS,Rosso} and by Green~\cite{Green} in a quite general setting (we refer the reader to the cited articles for details).
For our shuffle operation, the proposed generalization has good properties: this product is associative and verifies in particular a relation of the form $ \langle A^*\shuffle_q u, w\rangle = \qbin{w}{u}$.
For the infiltration product, we consider a family of definitions; we show that an analogue to the Chen--Fox--Lyndon relation~\eqref{eq:cfl}, see \cite[Chap.~6]{Lothaire1997},  cannot be attained.

We conclude this paper with \cref{sec:pgroup}. Let $p$ be a prime.
Since $q$-binomials are polynomials in $\mathbb{Z}[q]$, their coefficients may be reduced modulo~$p$ and then, these polynomials may be reduced modulo some non zero polynomial~$\mathfrak{M}$ in $\mathbb{F}_p[q]$. The languages of the form
\[
L_{v,\mathfrak{R},\mathfrak{M}}:=\left\{u\in A^*\mid \qbin{u}{v}\equiv \mathfrak{R}\pmod{\mathfrak{M}}\right\}
\]
where $\mathfrak{R}\in\mathbb{F}_p[q]$ is a non-constant polynomial of degree less than $\deg(\mathfrak{M})$ play a central role in our paper.
Linked to the notion of \emph{$k$-binomial equivalence} \cite{Eilenberg1976,RigoSalimov} (see \cite{LejeuneRosenfeld} for relationship with nilpotent groups), we say that two finite words $w_1,w_2\in A^*$ are {\em $(u,\mathfrak{M})$-binomially equivalent}  whenever, for all factors $v$ of $u$, 
    \[
    \qbin{w_1}{v}\equiv\qbin{w_2}{v}\pmod{\mathfrak{M}}.
    \]
    We study this equivalence relation and determine when it is a congruence and when the corresponding quotient of $A^*$ is a group. With \cref{thm:group_qinvertible} we get information about its order. From this, we generalize Eilenberg's theorem: a language is a $p$-group language if and only if it is a Boolean combination of languages of the form $L_{v,\mathfrak{R},a(q-1)^d}$. Indeed, our construction shows that the languages occurring in the classical formulation of Eilenberg's result are a disjoint union of some languages $L_{v,\mathfrak{R},a(q-1)^d}$.
    
\section{Introducing \texorpdfstring{$q$}{q}-deformations}\label{sec:2}

We now define the main object of the paper. We do this by generalizing the following Pascal--like formula which holds for the usual binomial coefficients of words: for all words $u$,
$v \in A^*$ and letters $a,b \in A$, we have
\begin{equation}\label{eq:conventional-pascal}
\binom{ua}{vb} = \binom{u}{vb} + \delta_{a,b}\binom{u}{v},
\end{equation}
where $\delta_{a,b}$ is the Kronecker delta on the letters of $A$.
For a reference on binomial coefficients of words, see \cite{Lothaire1997}. 
It is therefore natural to propose the following definition.
\begin{definition}
  We recursively define the {\em $q$-deformation}
  $\qbin{\cdot}{\cdot}$ --- an element of $\N[q]$ --- of the binomial coefficients on $A^* \times A^*$ as follows. For all words $u$,
  $v \in A^*$ and letters $a,b \in A$:
  \begin{equation}
    \label{eq:recdef}
    \qbin{u}{\varepsilon} = 1, \quad \qbin{\varepsilon}{v} = 0 \text{ if }v \neq \varepsilon, \quad \text{and } \quad \qbin{ua}{vb} = \qbin{u}{vb}\cdot q^{|vb|} + \delta_{a,b}\qbin{u}{v}.
  \end{equation}
\end{definition}

We note that the polynomial $\qbin{u}{v}$ evaluated at $1$ gives back the usual binomial coefficient 
$\binom{u}{v}$; this follows immediately from \eqref{eq:conventional-pascal} and the definition above.
Let $k,\ell$ be integers. In the case of a unary alphabet, we have
\[\qbin{a^k}{a^\ell} =\qbin{k}{\ell} \]
where on the right-hand side we have the Gaussian binomial coefficient of
integers. This easily follows from the fact that an analogue of Pascal's
identity exists:
\begin{equation}\label{eq:q-bin-Pascal}
\qbin{k+1}{\ell+1} = \qbin{k}{\ell+1}\cdot q^{\ell+1} + \qbin{k}{\ell}.
\end{equation}
  For arbitrary alphabets we observe that if $|u|<|v|$, then $\qbin{u}{v} = 0$; indeed this follows from the given definition straightforwardly. Moreover, if $|u| = |v|$ then $\qbin{u}{v} = 0$ if and only if $u\neq v$.
Furthermore, we have $\qbin{u}{u}=1$ for all words $u$. In fact, using these observations as the base cases for a double induction
on $|u|$ and $|v|$, we have
\begin{lemma}\label{lem:q-coefficient-zero}
For $u,v \in A^*$, we have that $\qbin{u}{v}$ is the zero polynomial if and only if $v$ does not appear as a subword in $u$.
\end{lemma}

\begin{remark}
  Instead of processing letters starting from the right, we could replace \cref{eq:recdef} with
  \[
  \qbin{au}{bv} = \qbin{u}{bv}\cdot q^{|bv|} + \delta_{a,b}\qbin{u}{v}
  \]
  and this leads to another polynomial with similar properties. As an example, take $u=01101$ and $v=01$. With \cref{eq:recdef}, the $q$-binomial is $q^6+q^5+q^3+1$ and if we process letters from the left, we would get $q^6+q^3+q+1$. We refer the reader to \cref{cor:reversal} where we deal with reversals.
\end{remark}

\begin{example}\label{exa:first}
 Let us apply the above definition to get 
  \[\qbin{0100110}{011}=q^{10}+q^9+q^6+q^4+q^3.\]
The computation is depicted in \cref{fig:tree}. If a node has two children, they correspond to the two terms in \eqref{eq:recdef}.  In that case, the multiplication by the convenient power of the variable $q$ is the label of the left edge. If a node has a single child, this means that the last letter of the two arguments are distinct. The curly bracket on a leaf is obtained by multiplying the labels of the edges going back to the root.%

  \begin{figure}[h!tb]
  \centering
  \begin{tikzpicture}
[
    level 2/.style = {sibling distance = 6cm},
    level 3/.style = {sibling distance = 4cm},
    level 4/.style = {sibling distance = 3cm},
    level 5/.style = {sibling distance = 1.5cm}
] 
\node {$\qbin{0100110}{011}$}
child {node {$\qbin{010011}{011}$}
  child {node {$\qbin{01001}{011}$}
    child {node {$\qbin{0100}{011}$}
      child {node {$\qbin{010}{011}=0$} edge from parent node [left] {$q^3$}}
      edge from parent node [left] {$q^3$}}
   child {node {$\qbin{0100}{01}$}
     child {node {$\qbin{010}{01}$}
       child {node {$\underbrace{\qbin{01}{01}=1}_{q^{10}}$} edge from parent node [left] {$q^2$}}
       edge from parent node [left] {$q^2$}}}
    edge from parent node [left] {$q^3$}}
  child {node {$\qbin{01001}{01}$}
    child {node {$\qbin{0100}{01}$}
      child {node {$\qbin{010}{01}$}
        child {node {$\underbrace{\qbin{01}{01}=1}_{q^9}$}
        edge from parent node [left] {$q^2$}}
        edge from parent node [left] {$q^2$}}
      edge from parent node [left] {$q^2$}}
    child {node {$\qbin{0100}{0}$}
      child {node {$\qbin{010}{0}$}
        child {node {$\qbin{01}{0}$}
          child {node {$\underbrace{\qbin{0}{0}=1}_{q^6}$}
          edge from parent node [left] {$q$}}
          edge from parent node [left] {$q$}}
        child {node {$\underbrace{\qbin{01}{\varepsilon}=1}_{q^4}$}}
      edge from parent node [left] {$q$}}
      child {node {$\underbrace{\qbin{010}{\varepsilon}=1}_{q^3}$}}
    }}
  edge from parent node [left] {$q^3$}};
\end{tikzpicture}
  \caption{Applying recursively the definition \cref{eq:recdef} to compute $\qbin{0100110}{011}$.}\label{fig:tree}
\end{figure}
Evaluating the above expression at $q=1$, we get the value of the classical binomial coefficient of words
\[\binom{0100110}{011}=5.\]
The non-zero coefficients of the $q$-deformation provide extra information about the occurrences of subwords as we will see in the next section.
\end{example}

\section{Combinatorial interpretation}\label{sec:combinter}

The exponents occurring in a $q$-deformed binomial provide extra data compared with the classical binomial coefficients of words. Each occurrence of $v$ as a subword of $u$ provides a term of the form $q^\alpha$ in $\qbin{u}{v}$ where $\alpha$ is the sum over all letters of $v$ of the number of letters at the right of them and not being part of that specific occurrence of the subword $v$.
The following theorem makes this observation precise.

\begin{theorem}\label{thm:powers} Let $u$ be a word over $A$, $k\ge 0$, and $a_1,\ldots,a_k\in A$. Then
  \[\qbin{u}{a_1\cdots a_k}=\sum_{\substack{u_0,u_1,\ldots,u_k\in A^*\\ u=u_0a_1\cdots u_{k-1}a_k u_k}} q^{\sum_{i=1}^k i|u_i|}.\]
\end{theorem}
As an example, we get
\[
\qbin{0110011}{01}=q^{10}+q^9+q^6+q^5+q^3+2 q^2+q,
\]
and we have to consider all factorizations of $0110011$ of the form $u_0\, 0\, u_1\, 1\, u_2$:
\[
\begin{array}{c|cccccccc}
  (u_1,u_2) & (\varepsilon,10011)& (1,0011) & (1100,1)&(11001,\varepsilon)&( 0,1)&( 01,\varepsilon)&( \varepsilon,1)&( 1,\varepsilon)\\
    \hline
    |u_1|+2 |u_2| & 10 & 9 & 6 & 5 & 3 & 2 & 2 & 1 \\
  \end{array}
  \]

  \begin{proof}
    The theorem is true for $k=0$, because the summation has only one term, namely $q^0$. The theorem is also true for $u=\varepsilon$ and any $k\geq 1$ as the sum is empty. Assume the theorem holds true for some $u\in A^*$ and all $k \geq 0$. Let $b \in A$ and
    $a_1$, \ldots, $a_k \in A$ with $k \geq 1$. Then
    \begin{align*}
      \sum_{\substack{u_0,u_1,\ldots,u_k\in A^*\\ ub=u_0a_1\cdots u_{k-1}a_k u_k}} \!\!\!\!\! q^{\sum_{i=1}^k i|u_i|}
      &=  \sum_{\substack{u_0,u_1,\ldots,u_k\in A^*, u_k\neq\varepsilon \\ ub=u_0a_1\cdots u_{k-1}a_k u_k}} q^{\sum_{i=1}^k i|u_i|} +  \sum_{\substack{u_0,u_1,\ldots,u_{k-1}\in A^* \\ ub=u_0a_1\cdots u_{k-1}a_k}} q^{\sum_{i=1}^{k-1} i|u_i|}\\
      &=  \sum_{\substack{u_0,u_1,\ldots,u_{k-1},u_k'\in A^*\\ u=u_0a_1\cdots u_{k-1}a_k u_k'}} \!\! q^{k+\sum_{i=1}^k i|u_i|} +\delta_{a_k,b}  \sum_{\substack{u_0,u_1,\ldots,u_{k-1}\in A^* \\ u=u_0a_1\cdots u_{k-1}}} q^{\sum_{i=1}^{k-1} i|u_i|}\\
      &= \qbin{u}{a_1\cdots a_k}q^k+\delta_{a_k,b} \qbin{u}{a_1\cdots a_{k-1}}=\qbin{ub}{a_1\cdots a_k}.
    \end{align*}
    To proceed from the first to the second line, in the considered factorizations of $ub$, note that $u_k$ ending with $b$ is of the form $u_k'b$ and this allows the first sum to be rewritten accordingly.  Hence the theorem holds true for $ub$.
  \end{proof}

We get the following straightforward corollary.

\begin{corollary}\label{cor:degg}
Let $u,v$ be words.  If $\binom{u}{v} \neq 0$, consider the leftmost occurrence of $v= a_1\cdots a_k$, $a_i \in A$ in $u = u_0a_1\cdots u_{k-1} a_k u_k$, $u_i \in A^*$, i.e., $u_0a_1\cdots u_{j-1}a_j$ is the shortest prefix of $u$ containing $a_1\cdots a_j$ for each $j=1,\ldots,k$. Then the polynomial $\qbin{u}{v}$ has degree $\sum_{i=1}^{k} i|u_i|$. Furthermore, it  is monic and the non-zero coefficient of the monomial of least degree is $1$.

In particular, the independent (constant) term $\qbin{u}{v}(0)$ equals $1$ if and only if $v$ is a suffix of $u$; otherwise it equals $0$. Similarly 
the degree of $\qbin{u}{v}$ is less than or equal to $|v|(|u|-|v|)$ and the coefficient of the monomial $q^{|v|(|u|-|v|)}$ is $1$ if and only if $v$ is a prefix of $u$; otherwise it equals $0$.
\end{corollary}

\begin{proof}
  Let $u=u_0'a_1 \cdots u_{k-1}'a_ku_k'$ be a factorization such that the expression $\sum_{i=1}^k i|u_i'|$ is maximal. We show that $u_i=u_i'$ for all $i\in\{0,\ldots,k\}$. Hence there is a unique occurrence of $v$ providing a monomial of highest degree.

  Clearly, $|u_0|=|u_0'|$ because if $|u_0'|>|u_0|$ then 
$u_0'a_1u_1'=u_0a_1u_1''$ for some $u_1''$ longer than $u_1'$ and $|u_1''|+\sum_{i=2}^k i|u_i'|>\sum_{i=1}^k i|u_i'|$ contradicting maximality. Proceed by induction. Assume $u_i=u_i'$ for all $i<\ell$ and prove similarly that $|u_\ell|=|u_\ell'|$.
  
  For the monomial of least degree, consider the rightmost occurrence of $v$. The second part of the statement is now immediate.
\end{proof}

In the next statement, in the $q$-analogue formula for binomial coefficients of words, we have here a correcting power of $q$ that does not appear in the classical setting of coefficients of words~\cite[Cor.~6.3.7]{Lothaire1997}. Such a correcting power already appears in the {\em $q$-Vandermonde identity} which is a $q$-analogue of the Vandermonde identity:
\[
\qbin{m + n}{k} =\sum_{j} q^{j(m-k+j)} \qbin{m}{k - j} \qbin{n}{j},
\]
where the nonzero contributions to this sum come from values of $j$ such that $\max(0, k-m)\le j \le min(n, k)$, see for instance \cite{Stanley}. The reader may also see \cite{machacek2023qrational} for the Chu--Vandermonde identity in the context of $q$-real binomial coefficients.

\begin{corollary}\label{cor:xy-u}
For all words $x,y \in A^*$ and letters $a_1,\ldots,a_k\in A$, we have
 \[
 \qbin{xy}{a_1\cdots a_k}=\sum_{j=0}^k q^{ j (|y|-k+j)} \qbin{x}{a_1\cdots a_j}\qbin{y}{a_{j+1}\cdots a_k}.
\]
\end{corollary}

\begin{proof}
  We apply \cref{thm:powers}:
  \begin{eqnarray*}
    \qbin{xy}{a_1\cdots a_k}&=& \sum_{j=0}^k \sum_{\substack{x=x_0a_1\cdots x_{j-1}a_jx_j,\\ y=y_0a_{j+1}\cdots y_{k-j-1}a_ky_{k-j}}} q^{\sum_{i=1}^ji|x_i|+\sum_{i=0}^{k-j}(i+j)|y_i|}\\
                            &=& \sum_{j=0}^k q^{j (|y|-k+j)} \sum_{x=x_0a_1\cdots x_{j-1}a_jx_j} q^{\sum_{i=1}^ji|x_i|} \sum_{ y=y_0a_{j+1}\cdots y_{k-j-1}a_ky_{k-j}} q^{\sum_{i=0}^{k-j}i|y_i|}
  \end{eqnarray*}
and the result follows.
\end{proof}

\begin{remark}\label{rem:xy-u}
  \cref{cor:xy-u} can also be written as follows. For all words $x,y,u \in A^*$, we have
 \[
 \qbin{xy}{u}=\sum_{\substack{u=u_1u_2\\ u_1,u_2 \in A^*}} q^{ |u_1| (|y|-|u_2|)} \qbin{x}{u_1}\qbin{y}{u_2}.
\]
\end{remark}

Proceeding by induction, we find

\begin{corollary}\label{cor:general}
Let $k\ge 2$ and $x_1,\ldots,x_k,u$ be words. We have
 \[
 \qbin{x_1\cdots x_k}{u}=\sum_{\substack{u=u_1\cdots u_k\\ u_i\in A^*}} q^{ \sum_{i=1}^{k-1} |u_i|(|x_{i+1}\cdots x_k|-|u_{i+1}\cdots u_k|)} \qbin{x_1}{u_1}\cdots \qbin{x_k}{u_k}.
 \]
\end{corollary}

We let $\widetilde{v}$ denote the reversal of the word $v$. In the next statement, $\binom{u}{v}_{1/q}$ means that the $q$-binomial polynomial is evaluated at $1/q$.
\begin{corollary}\label{cor:reversal}
  Let $u,v$ be words. We have 
  \[\qbin{u}{v}=q^{|v|(|u|-|v|)}\binom{\widetilde{u}}{\widetilde{v}}_{1/q}.\]
  In particular, if $u$ and $v$ are palindromes, then the list of coefficients of $\qbin{u}{v}$ for degrees ranging in $\{0,\ldots,|v|(|u|-|v|)\}$ is a palindrome.
\end{corollary}

\begin{proof}
  There is a bijection between the set of occurrences of the subword $v=a_1\cdots a_k$ in $u$ and the set of occurrences of $a_k\cdots a_1$ in $\widetilde{u}$. The factorization $u=u_0a_1\cdots u_{k-1}a_{k} u_k$ is associated with the factorization $\widetilde{u}=\widetilde{u_k}a_k\cdots \widetilde{u_1}a_1\widetilde{u_0}$. The latter factorization provides $\binom{\widetilde{u}}{\widetilde{v}}$ with a term $q$ with exponent
  \[\sum_{i=0}^k(k-i)|\widetilde{u}_i|=-\sum_{i=1}^k i|u_i|+k\sum_{i=0}^k |u_i|=-\sum_{i=1}^k i|u_i|+ |v|(|u|-|v|).\]
\end{proof}

\section{Some classical formulas revisited}\label{sec:5}

Dudik and Schulman \cite{Dudik} explicitly gave the following identity which is, for instance, useful to show $k$-binomial equivalence by only considering subwords of length exactly $k$ instead of length at most~$k$. Notice that it implicitly appears first in the work
of Manvel et al.~\cite[Lemma~1]{ManvelMSSS1991reconstruction}. If $|u|\ge k\ge |x|$,
then
\[\binom{|u|-|x|}{k-|x|} \binom{u}{x} = \sum_{t\in A^k}
\binom{u}{t}\binom{t}{x}.\] This relation nicely extends to
$q$-binomial coefficients of words.

\begin{theorem} Let $k$ be an integer, $u,x$ words such that $|u|\ge k\ge |x|$. We have 
\[\qbin{|u|-|x|}{k-|x|} \qbin{u}{x} = \sum_{t\in A^k} \qbin{u}{t}\qbin{t}{x}.\]
\end{theorem}

\begin{proof}
  Assume that $x$ occurs in $u$, otherwise both sides equal $0$. Fix a specific occurrence of $x=a_1\cdots a_\ell$ in~$u$ and a subword $t=b_1\cdots b_k$ of length $k$ of $u$ containing this particular occurrence of $x$.  Otherwise stated, we consider $u=u_0b_1\cdots u_{k-1}b_ku_k$ and indices $j_1<\cdots <j_\ell$ such that $b_{j_1}\cdots b_{j_\ell}=x$. We set $j_0=0$ and $j_{\ell+1}=k+1$. We focus on the contribution of these fixed elements as a power of $q$ to both sides of the relation. By \cref{thm:powers}, these occurrences respectively provide $\qbin{u}{t}$ and $\qbin{t}{x}$ on the right-hand side with
  \begin{equation}
    \label{eq:contri0}
    q^{\sum_{i=1}^k i |u_i|}\quad\text{ and }\quad q^{\sum_{i=1}^{\ell} i(j_{i+1}-j_i-1)}.
  \end{equation}
 Similarly, this occurrence of $x$ provides $\qbin{u}{x}$ on the left-hand side with
  \begin{equation}
    \label{eq:contri1}
    q^{\sum_{i=1}^{\ell} i(|u_{j_i}\cdots u_{j_{i+1}-1}|+j_{i+1}-j_{i}-1)}.
  \end{equation}
  So the difference of the exponents in \eqref{eq:contri0} and \eqref{eq:contri1} is equal to
  \begin{equation}
    \label{eq:ddsumi}
    \sum_{i=0}^\ell \sum_{n=j_i}^{j_{i+1}-1} (n-i)|u_n|.
  \end{equation}
  Consider the indices $m_1<\cdots <m_{k-\ell}$ such that $\{j_1,\ldots,j_\ell,m_1,\ldots,m_{k-\ell}\}=\{1,\ldots,k\}$. We set $m_0=0$ and $m_{k-\ell+1}=k+1$. This partition of $\{1,\ldots,k\}$ into two subsets of indices permits us to highlight either the letters of $x$ (when referring to the indices $j_i$) or the letters of $t$ and not in $x$ (using the indices $m_i$).

  We have reached a delicate point, and an example will give the reader a better understanding of the situation. Take $k=9$, $\ell=3$ and $m_1<m_2<j_1<m_3<j_2<m_4<m_5<j_3<m_6$ is the partition of $\{1,\ldots,9\}$. Using the indices in the previous sum, we have the following table
  \[\begin{array}{c|cccccccccc}
      & &m_1&m_2&j_1&m_3&j_2&m_4&m_5&j_3&m_6\\
      \hline
      i&0&0&0&1&1&2&2&2&3&3\\
      n&0&1&2&3&4&5&6&7&8&9\\
      (n-i)&0&1&2&2&3&3&4&5&5&6\\
      \end{array}\]

    In \cref{eq:ddsumi} we understand that $i$ is incremented each time a new $j_i$ index is reached and the value of $(n-i)$ is incremented each time a new $m_i$ index is reached. Hence \cref{eq:ddsumi} can be rewritten as
  \[\sum_{i=0}^{k-\ell}i|u_{m_i}\cdots u_{m_{i+1}-1}|\]
  which is the exponent of the term in the Gaussian coefficient (for the reader not used to these coefficients, we can interpret them as $q$-binomial coefficients of words over a unary alphabet $\{\diamond\}$ and use \cref{thm:powers})
  \begin{equation}
    \label{eq:gaussian_interpret}
    \qbin{|u|-\ell}{k-\ell}=\qbin{\diamond^{|u|-\ell}}{\diamond^{k-\ell}}=\sum_{\substack{w_0,w_1,\ldots,w_{k-\ell}\in \{\diamond\}^*\\
        \diamond^{|u|-\ell}=w_0\diamond\cdots w_{k-\ell-1}\diamond w_{k-\ell}}} q^{\sum_{i=1}^k i|w_i|}
  \end{equation}
  provided by the factorization \(|w_i|=|u_{m_i}\cdots u_{m_{i+1}-1}| \)
  \[ w_0\underline{\diamond}\cdots w_{k-\ell-1}\underline{\diamond} w_{k-\ell} = \diamond^{|u_0\cdots u_{m_1-1}|} \underline{\diamond} \cdots \diamond^{|u_{m_{k-\ell-1}}\cdots u_{m_{k-\ell}-1}|} \underline{\diamond}\, \diamond^{|u_{m_{k-\ell}}\cdots u_k|} \]  
where the underlined positions correspond exactly to the letters added to $x$ to get $t$.
\end{proof}

The following corollary generalizes an observation made in the introduction.
\begin{corollary}\label{cor:reconstruction}
Let $u \in A^*$ and $k\in\{1,\ldots, |u|\}$. The sequence $(\qbin{u}{x})_{x \in A^k}$ uniquely determines the word $u$.
\end{corollary}
\begin{proof}
  First note that $|u|$ can be determined from the largest degree in the sequence, see \cref{cor:degg}. 
The claim is true for $k=1$; indeed the polynomial $\qbin{u}{a}$ encodes the positions of $a$ in $u$.
For larger $k$, we use the above theorem to find the polynomials $\qbin{|u|-1}{k-1}\qbin{u}{a}$, $a \in A$, and thus the polynomials $\qbin{u}{a}$, which in turn uniquely determine
$u$.
\end{proof}

The classical formula
\[\sum_{v\in A^n}\binom{u}{v}=\binom{|u|}{n}\]
can naturally be extended to $q$-binomial coefficients of words. Observe that this result is independent of the size of the alphabet. 
If $P(q)=\sum_{i=0}^d c_i\, q^i$ is a polynomial in the variable~$q$, we let $[q^m]P$ denote the coefficient $c_m$ of the monomial of degree~$m$.

\begin{proposition}\label{prop:sum_over_subwords} Let $u$ and $v$ be words. Let $n\ge 0$ be an integer. We have
  \begin{equation}   \label{eq:sumu}
  \sum_{v\in A^n}\qbin{u}{v} = \qbin{|u|}{n}; \quad \text{and} \quad 
  \sum_{u\in A^n}\qbin{u}{v} = (\# A)^{n-|v|}\qbin{n}{|v|}
  \end{equation}
  where on the right-hand sides we have Gaussian binomial coefficients of integers. Consequently, for all~$i$
  \[
 [q^i]\qbin{u}{v}\le [q^i]\qbin{|u|}{|v|}.
 \]
\end{proposition}

Examples for both identities are
\begin{align*}
\sum_{v\in \{0,1\}^3}\qbin{011010}{v} &=\qbin{6}{3}=q^9+q^8+2 q^7+3 q^6+3 q^5+3 q^4+3 q^3+2 q^2+q+1; \text{ and}\\
\sum_{u\in \{0,1\}^5}\qbin{u}{01}     &=2^3\qbin{5}{2}=8 q^6+8 q^5+16 q^4+16 q^3+16 q^2+8 q+8.
\end{align*}
\begin{proof}
We first prove the identity on the left of \cref{eq:sumu}. We apply \cref{thm:powers} and since the sum ranges over all words of length $n$, this corresponds to selecting $n$ positions amongst the $|u|$ available ones
\begin{eqnarray*}
  \sum_{a_1\cdots a_n\in A^n}\qbin{u}{a_1\cdots a_n}&=&
                                                        \sum_{a_1\cdots a_n\in A^n}\sum_{\substack{u_0,u_1,\ldots,u_{n}\in A^* \\ u=u_0a_1\cdots u_{n-1} a_n u_{n}}} q^{\sum_{i=1}^n i|u_i|}\\
  &=& \sum_{1\le j_1<\cdots< j_n<|u|+1=j_{n+1}} q^{\sum_{i=1}^ni(j_{i+1}-j_i-1)}=\qbin{|u|}{n}.
\end{eqnarray*}
The last equality is the interpretation of Gaussian binomial coefficients as the one used in \cref{eq:gaussian_interpret}. 
   
The second identity of \cref{eq:sumu} is proved in a similar manner: by \cref{thm:powers}
\begin{eqnarray*}
  \sum_{u\in A^n}\qbin{u}{a_1\cdots a_k} &=& \sum_{u\in A^n}
                                             \sum_{\substack{u_0,u_1,\ldots,u_{k}\in A^* \\ u=u_0a_1\cdots u_{k-1} a_k u_{k}}} q^{\sum_{i=1}^k i|u_i|}\\
  &=& \sum_{j_0=0< j_1<\cdots< j_k<|u|+1=j_{k+1}}  \sum_{\substack{u_0,u_1,\ldots,u_{k}\in A^* \\ u=u_0a_1\cdots u_{k-1} a_k u_{k}\\ |u_i|=j_{i+1}-j_i-1\\}} q^{\sum_{i=1}^k i|u_i|}\\
  &=& (\# A)^{n-k} \sum_{j_0=0< j_1<\cdots< j_k<|u|+1=j_{k+1}} q^{\sum_{i=1}^k i(j_{i+1}-j_i-1)}.
\end{eqnarray*}
The third formula is a direct consequence of the first one and the fact that all coefficients of $q$-binomial coefficients of words are non-negative.
\end{proof}
  
\begin{remark}
The second identity of \cref{eq:sumu} also holds for classical coefficients of words (by letting $q$ tend to $1$) but seems to have been unnoticed in the literature.
The combinatorial explanation is the following one. Let $v$ be a word. If we list all words of length $n$ and fix exactly $|v|$ positions (amongst the $n$ available), there are exactly $(\# A)^{n-|v|}$ words of the list where $v$ occurs as a subword in these particular positions. Of course, notice that there are $\binom{n}{|v|}$ ways to choose $|v|$ positions amongst $n$.
\end{remark}

\section{Deformations in formal series}\label{sec:def}

Classical binomial coefficients of words appear naturally in certain operations on formal series.
Let $(\mathbb{K},+,\cdot,0,1)$ be a semiring such as $\mathbb{N}$ or $\mathbb{N}[q]$.
We let $\mathbb{K}\langle\langle A^*\rangle\rangle$ denote the set of {\em formal series} over $A^*$ with coefficients in $\mathbb{K}$, i.e., the set of maps from $A^*$ to $\mathbb{K}$. 

In this section we consider $q$-deformations of two operations, namely the \emph{shuffle} and \emph{infiltration} products.
These two operations are thoroughly studied in \cite[Chap.~6]{Lothaire1997}, and our aim is to define the corresponding $q$-deformations
in such a way that most of the nice properties, like associativity, extend to $q$-deformed versions. For the $q$-shuffle product
all the sought properties are established. For the $q$-infiltration, we consider a family \cref{eq:geninf} of definitions; each of them
is shown to fail to give an associative operation. However, focusing on one particular version, we establish $q$-deformed versions of some classical properties.

We set some terminology and notation used throughout this section. The coefficient of $w\in A^*$ in $s\in \mathbb{K}\langle\langle A^*\rangle\rangle$ is denoted by $\langle s,w\rangle$. The {\em support} of the series $s$ is the set of words $w$ such that $\langle s,w\rangle\neq 0$. The set of polynomials, i.e., series with finite support, is denoted by $\mathbb{K}\langle A^*\rangle$.  The {\em degree} of $s\in \mathbb{K}\langle A^*\rangle$ is the maximal length of the words with a non-zero coefficient.

\subsection{Deformed shuffle product}

The usual shuffle of two words belonging to $A^*$ is a polynomial in
$\mathbb{N}\langle A^*\rangle$ encoding the multiset of the words of the form $u_1v_1\cdots u_kv_k$,
where $u=u_1\cdots u_k$, $v=v_1\cdots v_k$, and $u_i,v_i\in A^*$ for all~$i$. For instance, shuffling
$010$ and $0$, we get
\[{\color{blue}010}\shuffle{\color{red}0}={\color{blue}010}{\color{red}0}+
{\color{blue}01}{\color{red}0}{\color{blue}0}+
{\color{blue}0}{\color{red}0}{\color{blue}10}+
{\color{red}0}{\color{blue}010}=2\cdot 0100+2\cdot 0010.\]
We define a $q$-deformed shuffle product of two words $u$ and $v$ by extending in a natural way the
recursive definition found in \cite[Chap.~6]{Lothaire1997}. We further inspect the corresponding properties of the $q$-shuffle, where the coefficients are not integers but polynomials in $q$.

To be meaningful we need to define two operations. The first one is an external operation of $\mathbb{N}[q]$ on $\mathbb{N}[q]\langle\langle A^*\rangle\rangle$ acting on the left (or on the right):
\[\forall P\in\mathbb{N}[q], s\in \mathbb{N}[q]\langle\langle A^*\rangle\rangle, w\in A^*,\quad \langle P.s,w\rangle=P\langle s,w\rangle.\]
The second one is an external operation of $A^*$ on $\mathbb{N}[q]\langle\langle A^*\rangle\rangle$ acting on the right:
\[\forall s\in \mathbb{N}[q]\langle\langle A^*\rangle\rangle, v,w\in A^*,\quad \langle s.v,wv \rangle=\langle s,w \rangle.\]
Considering $v$ as the series $1.v$, this is just the usual product of two series.

\begin{definition}\label{def:shuffle} Let $u,v$ be two finite words over an alphabet $A$ and $a,b\in A$.
  The {\em $q$-shuffle} of $u$ and $v$ is a polynomial in $\mathbb{N}[q]\langle A^*\rangle$ defined recursively by
  \begin{equation}
    u \shuffle_q \varepsilon = \varepsilon \shuffle_q u = 1.u; \quad 
    \label{eq:qshuffle}
    ua \shuffle_q vb = q^{|vb|} (u\shuffle_q vb)a+(ua\shuffle_q v)b.
  \end{equation}
\end{definition}
As an example, we have
\begin{equation}
  \label{eq:exashuffle}
  010 \shuffle_q 0 = (q^3+q^2) 0010 + (q+1) 0100,
\end{equation}
and the reader may notice that evaluation at $q=1$ gives back the classical shuffle.

\begin{lemma}\label{lem:combshuffle}
  Let $u$ be a word and $a_1,\ldots,a_n$ be letters. We have
  \[u \shuffle_q a_1\cdots a_n=\sum_{\substack{u_0,\ldots,u_n\in A^*\\ u=u_0\cdots u_n}} q^{\sum_{i=1}^n i|u_i|} u_0a_1\cdots u_{n-1}a_nu_n.\]
  In particular, $\deg \langle u\shuffle_q v,vu\rangle=|u|\, |v|$ is the maximal degree of the coefficients. Similarly, the constant term $1$ appears in the polynomial coefficient of $uv$.
\end{lemma}

\begin{proof}
  The support of the polynomial \(u \shuffle_q a_1\cdots a_n\) is the set
  \[\{u_0a_1\cdots u_{n-1}a_nu_n \mid u_0,\ldots,u_n\in A^*,\ u=u_0\cdots u_n\}.\]
Each factorization of $u=u_0\cdots u_n$ provides a power of $q$ by recursive applications of \eqref{eq:qshuffle}. We process each letter of $u_n$ with the first term of \eqref{eq:qshuffle}. This gives a factor $q^{n|u_n|}$. Then we consider the letter $a_n$ with the second term of \eqref{eq:qshuffle}. We continue with the letters of $u_{n-1}$ providing this time $q^{(n-1)|u_{n-1}|}$ and so on and so forth.
\end{proof}

Applying this result to the same example as above, we have
\[{\color{blue}010}\shuffle_q{\color{red}0}= 1\cdot {\color{blue}010}{\color{red}0}+
q\, {\color{blue}01}{\color{red}0}{\color{blue}0}+
q^2\, {\color{blue}0}{\color{red}0}{\color{blue}10}+
q^3\, {\color{red}0}{\color{blue}010}.\]
Roughly speaking, one has to count for each red letter, how many blue letters are on the right of it. Summing up these numbers gives the exponent to take into account. Similarly, 
\[010 \shuffle_q 00 =  (1+q+q^2) 01000+ (q^2+2q^3+q^4) 00100+ (q^4+q^5+q^6) 00010\]
because
\begin{align*}
  {\color{blue}010} \shuffle_q {\color{red}00} = &\ 
1\cdot {\color{blue}010}{\color{red}00}+
q\, {\color{blue}01}{\color{red}0}{\color{blue}0}{\color{red}0}+
q^2\, {\color{blue}0}{\color{red}0}{\color{blue}10}{\color{red}0}+
q^3\, {\color{red}0}{\color{blue}010}{\color{red}0}+
q^2\, {\color{blue}01}{\color{red}00}{\color{blue}0}+\\
&\ q^3\, {\color{blue}0}{\color{red}0}{\color{blue}1}{\color{red}0}{\color{blue}0}+
q^4\, {\color{red}0}{\color{blue}01}{\color{red}0}{\color{blue}0}+
q^4\, {\color{blue}0}{\color{red}00}{\color{blue}10}+
q^5\, {\color{red}0}{\color{blue}0}{\color{red}0}{\color{blue}10}+
q^6\, {\color{red}00}{\color{blue}010}.
\end{align*}

The following result gives an alternative statement for \cref{lem:combshuffle}. As a consequence, we get that the $q$-shuffle is associative. Let $\pi$ be a permutation in the symmetric group $S_n$. We let $\inv(\pi)$ be the number of inversions of $\pi$, i.e., $\inv(\pi)=\#\{1\le i<j\le n:\pi(i)>\pi(j)\}$.

\begin{corollary}\label{cor:ref} 
  Let $a_1,\ldots,a_n\in A$. We have
  \[ a_1\cdots a_k\shuffle_q a_{k+1}\cdots a_n=\sum_{\substack{\pi\in S_n: \pi(1)<\cdots <\pi(k),\\ \pi(k+1)<\cdots<\pi(n)}} q^{\inv(\pi)} a_{\pi(1)}\cdots a_{\pi(n)}\]
  and in particular, the $q$-shuffle is associative.
\end{corollary}

\begin{proof}
  Applying \cref{lem:combshuffle}, we get
  \[a_1\cdots a_k \shuffle_q a_{k+1}\cdots a_n=\sum_{\substack{u_0,\ldots,u_{n-k}\in A^*\\ a_1\cdots a_k=u_0\cdots u_{n-k}}} q^{\sum_{i=1}^{n-k} i|u_i|} u_0a_{k+1}\cdots u_{n-k-1}a_{n}u_{n-k}.\]
  Notice that $i|u_i|$ counts the number of inversions between the first $i$ letters $a_{k+1},\ldots,a_{k+i}$ and the letters of $u_i$.

  Let us consider associativity. For $0=k_0\le k_1\le \cdots \le k_h=n$, it is enough to observe that 
  \[a_{k_0+1}\cdots a_{k_1}\shuffle_q \cdots \shuffle_q a_{k_{h-1}+1}\cdots a_{k_h}=
    \sum_{\substack{\pi\in S_n:\\ \pi(k_i+1)<\cdots <\pi(k_{i+1})\\\text{ for }  0\le i<h}} q^{\inv(\pi)} a_{\pi(1)}\cdots a_{\pi(n)}. \]
\end{proof}

\begin{remark}\label{rem:ref}
  As pointed out by one of the reviewers, $(\mathbb{Z}[q]\langle A^*\rangle, \shuffle_q)$ is Green's quantized shuffle algebra of type $(A,\cdot)$ defined by $a\cdot b=1$ for all $a,b\in A$ over $\mathbb{Q}(q)$. See \cite[Sec.~4]{Green}. 
\end{remark}

The next lemma is similar to \cref{lem:combshuffle}, but here the focus is put on the letters of the first factor of the $q$-shuffle. In the above examples, it is equivalent to count for each blue letter, how many red letters are on its left.

\begin{lemma}\label{lem:combshuffle2}
Let $v$ be a word and $a_1,\ldots,a_k$ be letters. We have
\[a_1\cdots a_k\shuffle_q v=\sum_{\substack{v_0,\ldots,v_k\in A^*\\ v=v_0\cdots v_k}}  q^{\sum_{i=0}^k (k-i)|v_i|} v_0a_1v_1 \cdots v_{k-1}a_kv_k.\]
\end{lemma}

  \begin{proof}
    Proceed as in the proof of \cref{lem:combshuffle}. Assume that $v_0a_1v_1a_2v_2 \cdots v_{k-1}a_kv_k$ belongs to the support of $a_1\cdots a_k\shuffle_q v$, with $v=v_0\cdots v_k,\ v_i\in A^*$. Focus on a particular letter $a_i$. When applying \cref{eq:qshuffle}, this letter $a_i$ will contribute to the exponent of $q$ with the number of letters of $v$ to its left, which is exactly $|v_0\cdots v_{i-1}|$.
  \end{proof}
  
  We let $u\shuffle_{\frac{1}{q}} v$ denote the element in $\mathbb{N}[\frac{1}{q}]\langle\langle A^*\rangle\rangle$ such that $[q^{-i}]\langle u\shuffle_{\frac{1}{q}} v,w\rangle=[q^i]\langle u\shuffle_q v,w\rangle$ for all $i,w$. 
As an example,
\[010 \shuffle_{\frac{1}{q}} 0 = \left({\frac{1}{q^3}}+\frac{1}{q^2}\right) 0010 + \left(\frac{1}{q}+1\right) 0100.\]
\begin{proposition}\label{pro:reciprocal}
  The $q$-shuffle satisfies the reciprocal relation
  \[q^{|u| \cdot |v|} (v\shuffle_{\frac{1}{q}} u)=u\shuffle_q v.\]
\end{proposition}

\begin{proof}
We can proceed by induction on $|u| + |v|$. Or, it can be seen as a consequence of the previous two \cref{lem:combshuffle,lem:combshuffle2}. The $q$-shuffles $u\shuffle_{q} v$ and $v\shuffle_{q} u$ clearly have the same support. Let $w$ be a word in this support and consider the specific factorization $w=v_0a_1v_1a_2v_2\cdots v_{k-1}a_kv_k$ with $u=a_1\cdots a_k,\ a_i\in A$ and $v=v_1\cdots v_k,\ v_i\in A^*$. For each $a_i$, the contribution to $u\shuffle_{q} v$ given by \cref{lem:combshuffle2} is $q^{|v_0\cdots v_{i-1}|}$ and the one to $v\shuffle_{q} u$ given by \cref{lem:combshuffle} is \[q^{|v_i\cdots v_k|}=q^{|v|-|v_0\cdots v_{i-1}|}.\]
Finally, with $k=|u|$, 
\[\prod_{i=1}^k q^{|v|-|v_0\cdots v_{i-1}|}= q^{|u|\, |v|} \prod_{i=1}^k \left(\frac{1}{q}\right)^{|v_0\cdots v_{i-1}|}.\] 
\end{proof}

To conclude this section, we can extend \cref{def:shuffle} to the $q$-shuffle of a word $u\in A^*$ and a series $s\in\mathbb{N}[q]\langle\langle A^*\rangle\rangle$ by
\[u\shuffle_q s= \sum_{v\in A^*} \langle s,v\rangle\, (u\shuffle_qv) \quad\text{ and }\quad s\shuffle_q u= \sum_{v\in A^*} \langle s,v\rangle\, (v\shuffle_qu).\]
In particular, if $A^*$ is understood as the characteristic formal series whose coefficients are all equal to $1$, then 
\[u\shuffle_q A^*=u\shuffle_q \sum_{v\in A^*}v= \sum_{v\in A^*} u\shuffle_qv \quad\text{ and }\quad  A^* \shuffle_q u=\sum_{v\in A^*} v\shuffle_qu.\]
By bilinearity, the $q$-shuffle can be readily extended to two series in $\mathbb{N}[q]\langle\langle A^*\rangle\rangle$.

As a consequence of \cref{thm:powers} and \cref{lem:combshuffle}, we are now able to prove the following result, which directly links the notions of $q$-shuffle and $q$-deformed binomial coefficients of words.

\begin{proposition}\label{prop:shuffbin}
  Let $u,w$ be words. Then
  \begin{equation}
  \label{eq:coeff_shuffle}
  \langle A^*\shuffle_q u, w\rangle = \qbin{w}{u}.
\end{equation}
\end{proposition}

\begin{proof}
  If $u$ is not a subword of $w$, then both sides of \cref{eq:coeff_shuffle} are zero. We can now assume that $u$ appears at least once as a subword of $w$. Let us fix one particular occurrence of $u=a_1\cdots a_k$ in $w$:
  \[w= w_0 a_1 w_1 \cdots w_{k-1} a_k w_k\]
  where $a_i\in A$ and $w_i\in A^*$ for all $i$. Let $x=w_0w_1\cdots w_k$. The word $w$ appears in the support of $x\shuffle_q u$. From \cref{lem:combshuffle}, this particular occurrence of $u$ contributes to the coefficient of $w$ in $x\shuffle_q u$ with a monomial $q^{\sum_{i=1}^k i|w_i|}$, which is enough to conclude with the proof using \cref{thm:powers}.
\end{proof}

To illustrate the previous proposition, given the $q$-shuffles $01 \shuffle_q 0=(q^2+q)001+1\cdot 010$ and $10 \shuffle_q 0=q^2 010+(q+1)100$, we get
\[
\langle A^*\shuffle_q 0, 010\rangle = \qbin{010}{0}=1+q^2.
\]

In view of \cref{pro:reciprocal}, if we permute the factors of the $q$-shuffle, \cref{eq:coeff_shuffle} becomes
\[\langle u\shuffle_q A^*, w\rangle = q^{|u|\cdot(|w|-|u|)}\binom{w}{u}_{1/q}.\]

\subsection{Deformed infiltration}
Inspired by \cite{Duchamp2017}, where the authors investigate various products defined by a recurrence relation, one could define several $q$-infiltrations using recurrences of the form $u\uparrow_q \varepsilon = \varepsilon \uparrow_q u = 1.u$ and 
\begin{equation}
  \label{eq:geninf}
 ua \uparrow_q vb= q^{|vb|}(u\uparrow_q vb)a+(ua\uparrow_q v)b+q^{\alpha(ua,vb)}\, \delta_{a,b}(u\uparrow_q v)a
\end{equation}
for every choice of a map $\alpha:A^*\times A^*\to\mathbb{N}$. In  \cite{Duchamp2001}, the choice is $\alpha(ua,vb)$ to be the constant function $1$. Roughly speaking, with such a choice, when building a term $q^n w$ of a monomial in the infiltration product, every application of the third term will increase the exponent of $q^n$ by $1$. So it will only record how many times the third term was used (see \cref{rem:alpha1}).

As for the $q$-shuffle, we can extend the above definition to the $q$-infiltration of a word $u\in A^*$ and a series $s\in\mathbb{N}[q]\langle\langle A^*\rangle\rangle$ by
\[u\uparrow_q s= \sum_{v\in A^*} \langle s,v\rangle\, (u\uparrow_qv)\quad\text{ and }\quad
s\uparrow_q u= \sum_{v\in A^*} \langle s,v\rangle\, (v\uparrow_qu).\]
For two series $s,t\in\mathbb{N}[q]\langle\langle A^*\rangle\rangle$, we define
\[s\uparrow_q t= \sum_{u,v\in A^*} \langle s,u\rangle\, \langle t,v\rangle\, (u\uparrow_qv).\]
As an example,
\[010 \uparrow_q 0 = (q+1) 0100 + (q^3+q^2) 0010 + (q^{\alpha(010,0)}+q^{\alpha(0,0)})\, 010\]
because, we have
\[{\color{blue}010}\uparrow_q{\color{red}0}=1\cdot {\color{blue}010}{\color{red}0}+
q\, {\color{blue}01}{\color{red}0}{\color{blue}0}+
q^2\, {\color{blue}0}{\color{red}0}{\color{blue}10}+
q^3\, {\color{red}0}{\color{blue}010}+ q^{\alpha(010,0)}\, 01{\color{cyan}0}+ q^{\alpha(0,0)}\, {\color{cyan}0}10\]
where the cyan color is used when letters are merged when using the third term in \cref{eq:geninf}.  The reader may notice on this example, that the $q$-infiltration has terms of maximal degree equal to the $q$-shuffle. This directly follows from the definition \cref{eq:geninf} whose first two terms in the recurrence relation coincide with the $q$-shuffle. We state this observation for our record. 

Recall that the degree of a polynomial $P\in \mathbb{K}\langle A^*\rangle$ is the maximal length of the words with a non-zero coefficient.
\begin{lemma}\label{lem:infiltration_main}
  Let $u,v$ be words. There exists a polynomial $P_{u,v}$ of degree less than $|u|+|v|$ such that
  \[u\uparrow_q v= u\shuffle_q v + P_{u,v}.\]
  Moreover, $P_{u,v}=0$ if and only if $u$ and $v$ are over disjoint alphabets.
\end{lemma}

As we wish to recover the terms of~$\shuffle_q$ as defined in \cref{def:shuffle}, we do not have much freedom on the first two terms of the recurrence relation. Even though this first result seems promising, we start this section with a negative statement. Whatever the choice made for the function $\alpha$, the $q$-infiltration product is not associative and the Chen--Fox--Lyndon relation does not hold. In a second step, we focus on a specific choice of $\alpha(ua,vb)=|vb|$.

\begin{proposition}
  The $q$-infiltration product \eqref{eq:geninf} is not associative.
\end{proposition}
\begin{proof}
  Take $f=01$, $g=0$, $h=01$. The polynomial coefficients of $01$ in $(f\uparrow_q g)\uparrow_q h$ and $f\uparrow_q(g\uparrow_q h)$ are respectively
  $$q^{1+2\alpha (0,0)+\alpha (01,01)} \text{ and }
  q^{2\alpha (0,0)+\alpha (01,01)}.$$
  This leads to a contradiction.  
\end{proof}

For classical binomial coefficients of words, resulting of the associativity of the infiltration product, the following Chen--Fox--Lyndon relation holds true
\begin{equation}
  \label{eq:cfl}
  \forall f,g,h\in A^*,\quad \binom{h}{f}\binom{h}{g}=\sum_{w\in A^*} \langle f\uparrow g,w\rangle \binom{h}{w}.
\end{equation}
There is no such relation for $\uparrow_q$ defined by \cref{eq:geninf}.
As a counterexample, take $h=01010$, $f=010$, $g=1$. On the one hand, we have 
\[\qbin{h}{f}\qbin{h}{g}=q^9+2 q^7+2 q^5+2 q^3+q\]
which only contains odd powers of $q$. 
On the other hand, we get
\[ f\uparrow_q g=q^{1+\alpha(01,1)} 010 + q^3\, 1010+ (q + q^2)\, 0110 +  1\cdot 0101\]
and 
\[\sum_{w\in A^*} \langle f\uparrow_q g,w\rangle \qbin{h}{w}=2 q^3 + 2 q^4 + q^{1+\alpha(01,1)} (1+q^2+q^4+q^6)\]
which contains consecutive powers of $q$.

We now study the choice $\alpha(ua,vb)=|vb|$ in the definition to record some more information into the corresponding exponent.
\begin{definition} Let $u,v$ be two finite words over an alphabet $A$ and $a,b\in A$.
  The {\em $q$-infiltration} of $u$ and $v$ is defined recursively by
  \[u \Uparrow_q \varepsilon = \varepsilon \Uparrow_q u = 1.u; \quad 
    ua \Uparrow_q vb= (ua\Uparrow_q v)b +q^{|vb|}\left[ (u\Uparrow_q vb)a+\delta_{a,b}(u\Uparrow_q v)a\right].\]
  We use a specific symbol $\Uparrow_q$ to distinguish from the general situation \cref{eq:geninf}.
\end{definition}
Reconsidering the same example as above
\[010 \Uparrow_q 0 = (q+1) 0100 + (q^3+q^2) 0010 +  (q^3+q)\, 010.\]
\begin{remark}\label{rem:alpha1}
   Had we chosen the map $\alpha$ to be the constant $1$, we would obtain $2q$ as coefficient of $010$ just meaning that $010$ can be obtained in two ways applying once the third term of the recurrence.
  Here, the exponents $q^3$ and $q$ give the information where the merge of letters $0$ occurred.
\end{remark}

Recall that the valuation of a series $s\in \mathbb{K}\langle\langle A^*\rangle\rangle$,
denoted by $\val(s)$, is the least integer $n$ such that the support of $s$ contains a monomial of degree $n$; by convention $\val(0)=+\infty$.  If we permute the order of the two factors, the $q$-infiltration being non-symmetric, we may get a simpler expression. One can compare this result with \cref{pro:reciprocal}.

\begin{proposition}
  Let $u,v,w$ be words such that $|u|\ge |v|$.
  There exists a polynomial $R_{u,v}$ such that $\val(R_{u,v})>|u|$ and
 \[u\Uparrow_q v= q^{|v|(|v|+1)/2}\qbin{u}{v} u + R_{u,v}.\]
  There exists a polynomial $S_{v,u}$ such that $\val(S_{v,u})>|u|$ and 
   \[v\Uparrow_q u=  q^{|v|(|u|-\frac{|v|-1}{2})}  \binom{u}{v}_{1/q} u + S_{u,v}.\]
  In particular, if $|u|\ge |v|,|w|$, then
\[\langle (u\Uparrow_q v)\Uparrow_q w,u\rangle=q^{|w|(|w|+1)/2+|v|(|v|+1)/2}\qbin{u}{w}\qbin{u}{v}.\]
\end{proposition}

\begin{proof}
  The word $u$ appears in the support of $u\Uparrow_q v$ only if $v=a_1\cdots a_k$ is a subword of $u$, $a_i\in A$. Every other element in the support has length larger than $|u|$. We process one factorization of the form
  $u=u_0a_1\cdots u_{k-1}a_ku_k$ from right to left. From the definition of $\Uparrow_q$, each letter of $u_i$ provides a factor $q^{|a_1\cdots a_i|}$ and this is also the case with $a_i$. 
  So the coefficient of $u$ in $u\Uparrow_q v$ is given by
  \[
    \sum_{\substack{u_0,\ldots,u_k\in A^*\\
        u=u_0a_1\cdots u_{k-1}a_ku_k}}
    q^{\sum_{i=1}^k i(|u_i|+1)}.\]

  The reasoning is similar for $v\Uparrow_q u$. When processing one factorization of the form
  $u=u_0a_1\cdots u_{k-1} a_ku_k$ from right to left, only the letter $a_i$ provides a factor with exponent $|u|-|u_ia_{i+1}\cdots u_{k-1}a_ku_k|$. 
   So the coefficient of $u$ in $v\Uparrow_q u$ is given by
  \[
    \sum_{\substack{u_0,\ldots,u_k\in A^*\\
        u=u_0a_1\cdots u_{k-1}a_ku_k}}
    q^{k|u|-\sum_{i=1}^k i|u_i|-k(k-1)/2}= q^{k|u|-\frac{k(k-1)}{2}} \binom{u}{v}_{1/q} .\]
  \end{proof}

\section{Towards \texorpdfstring{$p$}{p}-groups}\label{sec:pgroup}

In this section, we introduce what we call the $(u,\mathfrak{M})$-binomial equivalence relation. In particular, we consider a congruence refining  it\footnote{Recall that an equivalence relation $\cong$ is a {\em refinement} of $\sim_{u,\mathfrak{M}}$ if, for all words $w_1,w_2$,
$w_1\cong w_2$ implies $w_1\sim_{u,\mathfrak{M}}w_2$.
We also say that $\sim_{u,\mathfrak{M}}$ is {\em coarser} than $\cong$.}. This way, the quotient of $A^*$ by such a congruence is a monoid. We further study this structure and determine when it is a group. Moreover, we obtain information about its order. As a corollary, we generalize Eilenberg's theorem characterizing $p$-group languages.

\begin{definition}\label{def:recognizable}
  The monoid $M$ {\em recognizes} the language $L\subseteq A^*$ if there exist a monoid morphism $\mu:A^*\to M$ and a subset $S$ of $M$ such that $L=\mu^{-1}(S)$. A language is {\em recognizable} if it is recognized by a finite monoid.
\end{definition}
It is a well-known result that a language is recognizable if and only if it is regular, i.e., there exists a \emph{deterministic finite automaton (DFA)} that recognizes the language \cite[Thm.~10.2.5]{Lawson2004}.

Let $p$ be a prime. A language recognized by a $p$-group is a {\em $p$-group language}. We follow Eilenberg's classical textbook \cite{Eilenberg1976} where it is shown that a language is a $p$-group language if and only if it is a Boolean combination of languages of the form
\begin{equation}
  \label{eq:pl1}
  L_{v,r,p}:=\left\{u\in A^*\mid \binom{u}{v}\equiv r\pmod{p}\right\}.
\end{equation}

Let $\mathbb{F}_p=\mathbb{Z}/p\mathbb{Z}$ be the field of integers modulo $p$ and let $\mathfrak{M}$ be a non-zero polynomial of degree~$d\ge 1$ in $\mathbb{F}_p[q]$. Since $q$-deformed binomial coefficients are polynomials in $\mathbb{Z}[q]$, their integer coefficients can be reduced modulo~$p$ to get a polynomial in $\mathbb{F}_p[q]$. Similarly to \cref{eq:pl1}, we will consider the languages of the form
\begin{equation}
  \label{eq:plang}
  L_{v,\mathfrak{R},\mathfrak{M}}:=\left\{u\in A^*\mid \qbin{u}{v}\equiv \mathfrak{R}\pmod{\mathfrak{M}}\right\},
\end{equation}
where $\mathfrak{R}\in\mathbb{F}_p[q]$ is a polynomial of degree less than $\deg(\mathfrak{M})$.

As a consequence of our treatment, we get a generalization of Eilenberg's theorem: a language is a $p$-group language if and only if it is a Boolean combination of languages of the form \cref{eq:plang} where $\mathfrak{M}=a(q-1)^d$ for some integer $d\ge 1$ and non-zero $a\in\mathbb{F}_p$.

In order to lighten the proofs of our main results, we gather all required preliminary definitions and results about polynomial algebra in the next subsection.

\subsection{Some polynomial algebra}\label{sec:pol.alg}

We let $\mathbb{K}$ denote the finite ring $\mathbb{F}_p[q]/\langle\mathfrak{M}\rangle$ of order $p^d$ and we let $\mathbb{K}^\star$ denote the
multiplicative group made of the units. Note that $\mathbb{K}$ is a field if and only if $\mathfrak{M}$ is irreducible over $\mathbb{F}_p$. As
$(\mathbb{K},\cdot)$ is a finite semigroup with identity~$1$, for each $\mathfrak{P}\in \mathbb{K}$, two elements of the sequence
$1,\mathfrak{P},\mathfrak{P}^{2},\ldots$ will be equal. Considering the first repetition, the least positive integers $i\ge 0$ and $k\ge 1$ such that
 \[
 \mathfrak{P}^{i} = \mathfrak{P}^{i+k}
\]
\begin{center}
  \begin{tikzpicture}[->,>=stealth',shorten >=1pt,auto,node distance=1.6cm,
                thick,main node/.style={}]

  \node[main node] (1) {1};
  \node[main node] (2) [right of=1, node distance=1.2cm] {$\mathfrak{P}$};
  \node[main node] (3) [right of=2, node distance=1.2cm] {$\mathfrak{P}^2$};
   \node[main node] (3b) [right of=3, node distance=.8cm] {$\cdots$};
  \node[main node] (4) [right of=3b] {$\mathfrak{P}^{i-1}$};
  \node[main node] (5) [right of=4] {$\mathfrak{P}^i$};
  \node[main node] (6) [above right of=5] {$\mathfrak{P}^{i+1}$};
  \node[main node] (7) [below right of=5] {$\mathfrak{P}^{i+k-1}$};
  \node[main node] (8) [below right of=6] {$\vdots$};
  \path
    (1) edge node {} (2)
    (2) edge node {} (3)
    (3b) edge node {} (4)
    (4) edge node {} (5)
    (5) edge [bend left] node {} (6)
    (6) edge [bend left] node {} (8)
    (8) edge [bend left] node {} (7)
    (7) edge [bend left] node {} (5);      
  \end{tikzpicture}
\end{center}
are respectively the \emph{index} and the \emph{period} of $\mathfrak{P}$ and are denoted by $\ind(\mathfrak{P})$ and $\per(\mathfrak{P})$ (the semigroup~$\mathbb{K}$ being clear from the context). Note that for an arbitrary semigroup (without identity element), the usual definition of index is to consider an integer $i\ge 1$ because $\mathfrak{P}^0$ is not well-defined. For our record, we make the following trivial observation:
\begin{fact}\label{fac:trivial}
For all $m,n\in \N$, we have $\mathfrak{P}^{m} = \mathfrak{P}^{n}$ in $\mathbb{K}$ if and only if $m=n$ or $m\equiv n\pmod{\per(\mathfrak{P})}$ with
$m,n \geq \ind(\mathfrak{P})$.
\end{fact}

In our setting, observe that $\ind(\mathfrak{P})=0$ if and only if $\mathfrak{P}$ is a unit. In particular, the index and period of the polynomial $q$ will have a central role in our results. The index of $q$ is directly linked to the polynomial $\mathfrak{M}$, as we shall see in what follows.

\begin{definition}\label{def:val}
The \emph{valuation} $\val(\mathfrak{P})$ of a polynomial $\mathfrak{P}\in\mathbb{F}[q]$, where $\mathbb{F}$ is a field, is the largest power of $q$ dividing $\mathfrak{P}$, i.e., $\val(\mathfrak{P}) = \max\{ \ell \colon q^{\ell} \vert \mathfrak{P}(q) \}$.
In particular, $\val(0)=+\infty$.
\end{definition}

The following result generalizes the previous lemma. We use again the notation $[q^i]P$ to denote coefficients of polynomials, see right before \cref{prop:sum_over_subwords}.

\begin{proposition}\label{prop:ind(q)}
In $\mathbb{F}_p[q]/{\langle\mathfrak{M}\rangle}$, we have $\ind(q) = \val(\mathfrak{M})$.
\end{proposition}
\begin{proof}
  If $\mathfrak{M}=a q^d$ with $a \neq 0$, then the sequence of powers of $q$ modulo $\mathfrak{M}$ is of the form $1,q,\ldots,q^{d-1},0,0,\ldots$ where the first $d$ terms are pairwise distinct. Hence $\ind(q)=d = \val(\mathfrak{M})$.
  
  Assume now that $\mathfrak{M}$ has degree~$d\ge 1$ and at least two non-zero terms. Observe that the map $\varphi:\mathfrak{P}\mapsto q \mathfrak{P}\pmod{\mathfrak{M}}$ is a bijection on the set $S=\left\{\sum_{i=\val(\mathfrak{M})}^{\deg(\mathfrak{M})-1} f_iq^i\mid f_{\val(\mathfrak{M})},\ldots,f_{\deg(\mathfrak{M})-1}\in\mathbb{F}_p \right\}$. Indeed, let $\mathfrak{P}\in S$. The remainder of the Euclidean division of $q\mathfrak{P}$ by $\mathfrak{M}$ is $q\mathfrak{P}-\frac{[q^{\deg{\mathfrak{M}}-1}]\mathfrak{P}}{[q^{\deg{\mathfrak{M}}}]\mathfrak{M}}\mathfrak{M}$. The inverse of $\varphi$ is given by $\mathfrak{P}\mapsto q^{-1} \left(\mathfrak{P}-\frac{[q^{\val(\mathfrak{M})}]\mathfrak{P}}{[q^{\val(\mathfrak{M})}]\mathfrak{M}}\mathfrak{M}\right)$. For all $j\ge\val(\mathfrak{M})$, $q^j\pmod{\mathfrak{M}}$ belongs to~$S$ (this follows directly from the fact that the map $\varphi$ maps $S$ to $S$). Hence, $q^i\not\equiv q^j\pmod{\mathfrak{M}}$ for $0\le i<\val(\mathfrak{M})\le j$. Finally, since $S$ is finite, there exist $i>j\ge 0$ such that $q^{i+\val(\mathfrak{M})}\equiv q^{j+\val(\mathfrak{M})} \pmod{\mathfrak{M}}$. Since 
  $\varphi$ is a bijection on $S$, $q^{i-j+\val(\mathfrak{M})}\equiv q^{\val(\mathfrak{M})} \pmod{\mathfrak{M}}$.
\end{proof} 
 
The following proposition will be important when discussing $p$-groups in the following subsection.

\begin{proposition}\label{prop:order_of_q_power_of_p}
The monomial $q$ has $\ind(q) = 0$ and $\per(q) = p^t$ for some $t>0$ in $\mathbb{F}_p[q]/\langle\mathfrak{M}\rangle$ if and only if $\mathfrak{M}$ is of the form $a(q-1)^d$ for some $a \in \mathbb{F}_p^\star$ and $d>0$.
\end{proposition}
\begin{proof}
First assume that $\mathfrak{M}(q)=a(q-1)^d$ for some positive integer $d$ and $a \in \mathbb{F}_p^\star$. Take $n>0$ such that $p^n\ge d$. Since the base field is $\mathbb{F}_p$, we have
\[
q^{p^n}-1=(q-1)^{p^n}=a^{-1}(q-1)^{p^n-d}a(q-1)^d=a^{-1}(q-1)^{p^n-d}\mathfrak{M}(q),
\]
so that $q^{p^n}\equiv 1 \pmod{\mathfrak{M}}$. Therefore $\ind(q)=0$ and $\per(q)$ divides $p^n$. 

Now assume that $\per(q)=p^t$ for some $t>0$ and $\ind(q) = 0$. There thus exists a polynomial $\mathfrak{A}$ such that
\[
\mathfrak{A}(q)\mathfrak{M}(q)=q^{p^t}-1=(q-1)^{p^t}.
\]
Since polynomial rings over fields are unique factorization domains, we conclude that $\mathfrak{M}(q)$ has the expected form.
\end{proof}

\subsection{On \texorpdfstring{$q$}{q}-binomials and \texorpdfstring{$p$}{p}-group languages}

 We now define an equivalence relation on $A^*$ as follows. 
If the word $w\in A^*$ can be written $pfs$ with $p,f,s\in A^*$, then $f$ is a {\em factor} of $w$. The set of factors of $w$ is denoted by $\Fac(w)$.
 \begin{definition}[$(u,\mathfrak{M})$-binomial equivalence]
   Let $u\in A^+$ and $\mathfrak{M}\in\mathbb{F}_p[q]$. Two finite words $w_1,w_2\in A^*$ are {\em $(u,\mathfrak{M})$-binomially equivalent} and we write $w_1\sim_{u,\mathfrak{M}} w_2$ whenever 
    \[ \forall v\in \Fac(u)\, :\, \qbin{w_1}{v}\equiv\qbin{w_2}{v}\pmod{\mathfrak{M}}.\]
  \end{definition}
      We stress that one has to consider all factors of $u$ and compute the corresponding $q$-binomial coefficients reduced modulo~$\mathfrak{M}$. There is no need to consider the trivial empty factor. Incidentally, the number of equivalence classes is bounded by $\#\mathbb{K}^{(\#\Fac(u)-1)}$. In particular, the equivalence relation
      $\sim_{u,\mathfrak{M}}$ has finite index in $A^*$.

\begin{remark}
  It is obvious that if $f$ is a factor of $u$, then $w_1\sim_{u,\mathfrak{M}} w_2$ implies $w_1\sim_{f,\mathfrak{M}} w_2$ because $\Fac(f)\subset\Fac(u)$. 
\end{remark}
  
\begin{example}\label{exa:equiv}
  Let $\mathfrak{M}(q)=q^2+1\in\mathbb{F}_2[q]$ and $u=01\in\{0,1\}^*$. Hence $\mathbb{K}=\{0,1,q,q+1\}$ and $\Fac(u)=\{\varepsilon,0,1,01\}$. By straightforward computation, we find $4^3=64$ equivalence classes---the maximal possible---with representatives of length at most $8$. For instance, the word $w=00011101$ is the shortest one such that
  \[\qbin{w}{0}\equiv q+1,\quad \qbin{w}{1}\equiv q+1,\quad \qbin{w}{01}\equiv q \pmod{q^2+1}.\]
\end{example}

  \begin{remark}\label{rem:restrict_def}
    The $(u,\mathfrak{M})$-binomial equivalence relation is not always a congruence on $A^*$. Take $p=2$, $\mathfrak{M}(q)=q^2+1$, $A=\{0,1\}$ and $u=0$. The simplest counterexample is $0\sim_{0,\mathfrak{M}}0$, $1\sim_{0,\mathfrak{M}}\varepsilon$ but $01\not\sim_{0,\mathfrak{M}}0$ because $\qbin{01}{0}\equiv q\not\equiv 1\equiv\qbin{0}{0}\pmod{\mathfrak{M}}$. 
  \end{remark}

In light of \cref{def:recognizable}, we seek for a congruence $\equiv$ on $A^*$ (of finite index) refining the equivalence relation~$\sim_{u,\mathfrak{M}}$
so that $A^*/{\equiv}$ is a finite monoid.

The following lemma gives some necessary conditions for a congruence to refine the $(u,\mathfrak{M})$-binomial equivalence.

\begin{lemma}\label{lem:coarsest}
  Let $u\in A^+$, $\mathfrak{M}\in\mathbb{F}_p[q]$ and $\cong$ be any congruence refining the $(u,\mathfrak{M})$-binomial equivalence~$\sim_{u,\mathfrak{M}}$. Let $w_1,w_2$ be words in $A^*$ such that $w_1 \cong w_2$. Then 
  \begin{itemize}
  \item $q^{|w_1|}\equiv q^{|w_2|}\pmod{\mathfrak{M}}$, and 
  \item $|w_1|=|w_2|$ or $\val\left(\qbin{w_i}{v}\right)+|w_i|-|v|\ge\val(\mathfrak{M})$ for all $i\in\{1,2\}$, $v\in A^*$ such that $av\in\Fac(u)$ for some $a\in A$. 
  \end{itemize}
\end{lemma}

\begin{proof}
  Let $v\in A^*$ such that $av\in\Fac(u)$ for some $a\in A$. By \cref{thm:powers}, for $i\in\{1,2\}$, we have
  \[
    \qbin{aw_i}{av}=\qbin{w_i}{av}+q^{|w_i|-|v|}\qbin{w_i}{v}.
  \]
  Since $\cong$  is a refinement of $\sim_{u,\mathfrak{M}}$, $\qbin{w_1}{av}\equiv\qbin{w_2}{av}\pmod{\mathfrak{M}}$. It is also congruence, so $aw_1 \cong aw_2$ and $\qbin{aw_1}{av}\equiv\qbin{aw_2}{av}\pmod{\mathfrak{M}}$. It follows that
  \begin{equation}
    \label{eq:jep1}
q^{|w_1|-|v|}\qbin{w_1}{v}\equiv q^{|w_2|-|v|}\qbin{w_2}{v}\pmod{\mathfrak{M}}    
  \end{equation}
 In particular, $v=\varepsilon$ gives that $q^{|w_1|}\equiv q^{|w_2|}\pmod{\mathfrak{M}}$ since $u\neq\varepsilon$.

    Note that for each $\mathfrak{P}\in\mathbb{F}_p[q]$ we have
    \begin{equation}\label{eq:jep3}
      \mathfrak{P}\equiv\sum_{i=0}^{\val(\mathfrak{M})-1} ([q^i]\mathfrak{P})q^i+\sum_{i=\val(\mathfrak{M})}^{\deg(\mathfrak{M})-1} c_iq^i\pmod{\mathfrak{M}}
    \end{equation}
    for some $c_i\in\mathbb{F}_p$ because by Euclidean division $\mathfrak{P}=\mathfrak{A}.\mathfrak{M}+(\mathfrak{P}\bmod{\mathfrak{M}})$ and $\val(\mathfrak{A}\mathfrak{M})\ge \val(\mathfrak{M})$. In particular, we have
    \begin{eqnarray}
      \val(\mathfrak{P}\bmod{\mathfrak{M}})=\val(\mathfrak{P}),&  \text{ if }0\le \val(\mathfrak{P})<\val(\mathfrak{M})\text{ and } \label{eq:jep2} \\
                                                                   \val(\mathfrak{P}\bmod{\mathfrak{M}})\ge \val(\mathfrak{M}),&\text{ if }\val(\mathfrak{P})\ge\val(\mathfrak{M}). \nonumber 
    \end{eqnarray}
    Assume that
    \[
      \val\left(\qbin{w_1}{v}\right)+|w_1|-|v|<\val(\mathfrak{M}).
    \]
    Plugging the polynomials from \cref{eq:jep1} into \cref{eq:jep3}, we get $\val\left(\qbin{w_1}{v}\right)+|w_1|-|v|=\val\left(\qbin{w_2}{v}\right)+|w_2|-|v|$. The assumption implies that $\val\left(\qbin{w_i}{v}\right)$ is finite, i.e., $\qbin{w_i}{v}\neq 0$ and thus $|w_i|\ge|v|$. Hence $\val\left(\qbin{w_i}{v}\right)<\val(\mathfrak{M})$. It follows from \eqref{eq:jep2} that $\val\left(\qbin{w_i}{v}\right)=\val\left(\qbin{w_i}{v}\bmod{\mathfrak{M}}\right)$. Since $\qbin{w_1}{v}\equiv \qbin{w_2}{v} \pmod{\mathfrak{M}}$, we therefore conclude that $|w_1|=|w_2|$. 
\end{proof}

\begin{proposition}\label{pro:q_not_unit}
    Let $u\in A^+$, $\mathfrak{M}\in\mathbb{F}_p[q]$ and $\cong$ be a congruence that refines $\sim_{u,\mathfrak{M}}$. If $q$ is not a unit of~$\mathbb{K}$, then the monoid $A^*/{\cong}$ is a not group. In particular, no element except for the identity is invertible.
\end{proposition}

\begin{proof}
  The class $[\varepsilon]$ for $\cong$ is clearly the identity element of the monoid. 
 Thanks to \cref{lem:coarsest}, if $\ind(q)\ge 1$, then $[\varepsilon]$ is a singleton. For any non-empty word $w$ and for all words $x$, the class $[w]\cdot [x]=[wx]$ differs from $[\varepsilon]$, thus $[w]$ is not invertible.
\end{proof}

Nevertheless, it is interesting to look for the coarsest congruence refining $\sim_{u,\mathfrak{M}}$. The following result, along with \cref{lem:coarsest}, enables us to state the conditions defining this congruence.

\begin{proposition}\label{prop:coarsest.cong}
Let $u\in A^+$ and $\mathfrak{M}\in\mathbb{F}_p[q]$. Define the relation $\equiv_{u,\mathfrak{M}}$ by
$w_1 \equiv_{u,\mathfrak{M}} w_2$ if 
 \begin{itemize}
 \item $w_1\sim_{u,\mathfrak{M}}w_2$, and
  \item $q^{|w_1|}\equiv q^{|w_2|}\pmod{\mathfrak{M}}$, and 
  \item $|w_1|=|w_2|$ or $\val\left(\qbin{w_i}{v}\right)+|w_i|-|v|\ge\val(\mathfrak{M})$ for all $i\in\{1,2\}$, $v\in A^*$ such that $av\in\Fac(u)$ for some $a\in A$. 
  \end{itemize}

Then $\equiv_{u,\mathfrak{M}}$ is a congruence. Moreover, it is the coarsest one refining $\sim_{u,\mathfrak{M}}$.
\end{proposition}
\begin{proof} 
Let $w_1,w_2,y_1,y_2$ be words such that $w_1\equiv_{u,\mathfrak{M}}w_2$ and $y_1\equiv_{u,\mathfrak{M}}y_2$, and let $v$ be a factor of~$u$. From \cref{rem:xy-u},
\begin{align*}
\qbin{y_1w_1}{v} & =\sum_{\substack{v=v_1v_2\\ v_1,v_2 \in A^*}} q^{ |v_1| (|w_1|-|v_2|)} \qbin{y_1}{v_1}\qbin{w_1}{v_2}\\
& = \qbin{w_1}{v} + q^{|v||w_1|}\qbin{y_1}{v} + \sum_{\substack{v=v_1v_2\\ v_1,v_2 \in A^+}} q^{ |v_1| (|w_1|-|v_2|)} \qbin{y_1}{v_1}\qbin{w_1}{v_2}.
\end{align*}
Since $v_1$ and $v_2$ are factors of $v\in\Fac(u)$, 
\begin{equation}\label{eq:jep4}
\qbin{y_1}{v_1}\equiv \qbin{y_2}{v_1},\quad  \qbin{w_1}{v_2}\equiv \qbin{w_2}{v_2} \pmod{\mathfrak{M}}.
\end{equation}
Moreover, we have
\[
q^{|v_1||w_1|}\equiv q^{|v_1||w_2|}\pmod{\mathfrak{M}}.
\]
For the sake of notation, let us define, for all words $v,w$,
\[
  l_{w,v}=\val\left(\qbin{w}{v}\bmod{\mathfrak{M}}\right).
\]
Let $R$ be the remainder of the Euclidean division of $\qbin{w_1}{v_2}$ by $\mathfrak{M}$, so that $R/q^{l_{w_1,v_2}}$ is a polynomial.
  Finally, we have
  \begin{multline*}
\sum_{\substack{v=v_1v_2\\ v_1,v_2 \in A^+}} \!\! q^{ |v_1| (|w_1|-|v_2|)} \qbin{y_1}{v_1}\qbin{w_1}{v_2}\\
\equiv \sum_{\substack{v=v_1v_2\\ v_1,v_2 \in A^+}} \!\! q^{(|v_1|-1)(|w_1|-|v_2|)+l_{w_1,v_2}+|w_1|-|v_2|} \qbin{y_1}{v_1}\frac{R}{q^{l_{w_1,v_2}}} \pmod{\mathfrak{M}}.
\end{multline*}
Using \eqref{eq:jep2} note that $\val\left(\qbin{w}{v}\right)+|w|-|v|\ge\val(\mathfrak{M})$ implies $\val\left(\qbin{w}{v}\bmod{\mathfrak{M}}\right)+|w|-|v|\ge\val(\mathfrak{M})$. Hence $l_{w_1,v_2}+|w_1|-|v_2|\geq\val(\mathfrak{M})=\ind(q)$, by assumption $|w_2|=|w_1|+d\per(q)$ for some $d$, we have
\begin{align*}
q^{(|v_1|-1)(|w_1|-|v_2|)+l_{w_1,v_2}+|w_1|-|v_2|} & \equiv q^{(|v_1|-1)(|w_1|-|v_2|)+l_{w_2,v_2}+|w_1|-|v_2|+d|v_1|\per(q)}\pmod{\mathfrak{M}}\\
& \equiv q^{|v_1|(|w_2|-|v_2|)+l_{w_2,v_2}}\pmod{\mathfrak{M}},
\end{align*}
with $l_{w_1,v_2}=l_{w_2,v_2}$ coming from \cref{eq:jep4}. In the end, we get
\[
\sum_{\substack{v=v_1v_2\\ v_1,v_2 \in A^+}} q^{ |v_1| (|w_1|-|v_2|)} \qbin{y_1}{v_1}\qbin{w_1}{v_2} \equiv \sum_{\substack{v=v_1v_2\\ v_1,v_2 \in A^+}} q^{ |v_1| (|w_2|-|v_2|)} \qbin{y_2}{v_1}\qbin{w_2}{v_2}\pmod{\mathfrak{M}},
\]
so that $\qbin{y_1w_1}{v}\equiv\qbin{y_2w_2}{v}$ for all factors $v$ of $u$, i.e., $y_1w_1\equiv_{u,\mathfrak{M}}y_2w_2$.

The fact that $\equiv_{u,\mathfrak{M}}$ is the coarsest congruence refining $\sim_{u,\mathfrak{M}}$ follows from
\cref{lem:coarsest}.
\end{proof}

\begin{remark}\label{rem:all_letters}
We provide two cases where the congruence defined in \cref{prop:coarsest.cong} has a simpler expression:

\textbf{Case 1.} The first case is when $q$ is a unit of $\mathbb{K}$. Here $\ind(q)=0=\val(\mathfrak{M})$. Consequently
\[
w_1 \equiv_{u,\mathfrak{M}} w_2 \quad \text{if and only if} \quad w_1\sim_{u,\mathfrak{M}}w_2 \text{ and } |w_1|\equiv |w_2| \pmod{\per(q)}.
\]

\textbf{Case 2.} The second case is when $u$ contains all letters of $A$. In this case, we show next, that this implies $q^{|w_1|}\equiv q^{|w_2|} \pmod{\mathfrak{M}}$. As a consequence,
in this case, we have
\begin{align*}
w_1 \equiv_{u,\mathfrak{M}} w_2 \quad \text{if and only if} &
\quad w_1\sim_{u,\mathfrak{M}}w_2 \text{ and }  |w_1|  = |w_2| \text{ or }\\
                                                            & \min_{\substack{i\in\{1,2\}, v\in A^*:\\ \exists a\in A, av\in\Fac(u)}}\left\{\val\left(\qbin{w_i}{v}\right)+|w_i|-|v|\right\}\ge\val(\mathfrak{M}) . 
\end{align*}
Since $u$ contains all letters of $A$, \cref{prop:sum_over_subwords} implies for $i=1,2$ 
\[
  \sum_{a\in A}\qbin{w_i}{a}=\qbin{|w_i|}{1}=\frac{q^{|w_i|}-1}{q-1}.
\]
Now $w_1\sim_{u,\mathfrak{M}}w_2$ implies that $\qbin{w_1}{a}\equiv \qbin{w_2}{a}\pmod{\mathfrak{M}}$ for all $a\in A$. Hence
\[\qbin{|w_1|}{1}\equiv\qbin{|w_2|}{1}\pmod{\mathfrak{M}}\]
thus $q^{|w_1|}-q^{|w_2|}=(q-1)\mathfrak{P}\mathfrak{M}$ for some polynomial $\mathfrak{P}$.

To conclude this remark, we combine the two above cases, namely, if both $q$ is invertible and $u$ contains all letters of $A$, we find that the congruence $\equiv_{u,\mathfrak{M}}$ is actually the equivalence relation $\sim_{u,\mathfrak{M}}$ itself (i.e., there is no proper refinement). 
\end{remark}

When $q$ is a unit of $\mathbb{K}$, for a congruence  $\cong$ that refines $\sim_{u,\mathfrak{M}}$, we may ask whether or not the monoid $A^*/{\cong}$ is a group. It is natural to consider the \emph{coarsest} congruence refining $\sim_{u,\mathfrak{M}}$, which is $\equiv_{u,\mathfrak{M}}$ in view of \cref{prop:coarsest.cong}. In particular, since we consider that $q$ is a unit of $\mathbb{K}$, \cref{rem:all_letters} shows that this congruence is actually the intersection of relations $\sim_{u,\mathfrak{M}}\cap \sim_{\per(q)}$ (or just $\sim_{u,\mathfrak{M}}$ if $u$ contains all letters of $A$) where $x \sim_{\per(q)} y$ if and only if $|x|\equiv |y| \pmod{\per(q)}$. Considering this particular congruence, the corresponding quotient is indeed a group as seen in \cref{thm:group_qinvertible}.

\begin{example}
  Continuing \cref{exa:equiv}, if we take for each equivalence the genealogically least representative and if the $64$ classes are ordered with respect to these representatives, we get the multiplication table depicted in~\cref{fig:mult_table}. The element in the upper left corner is $[\varepsilon]$.
\end{example}
\begin{figure}[h!tb]
  \centering
  \scalebox{.4}{\includegraphics{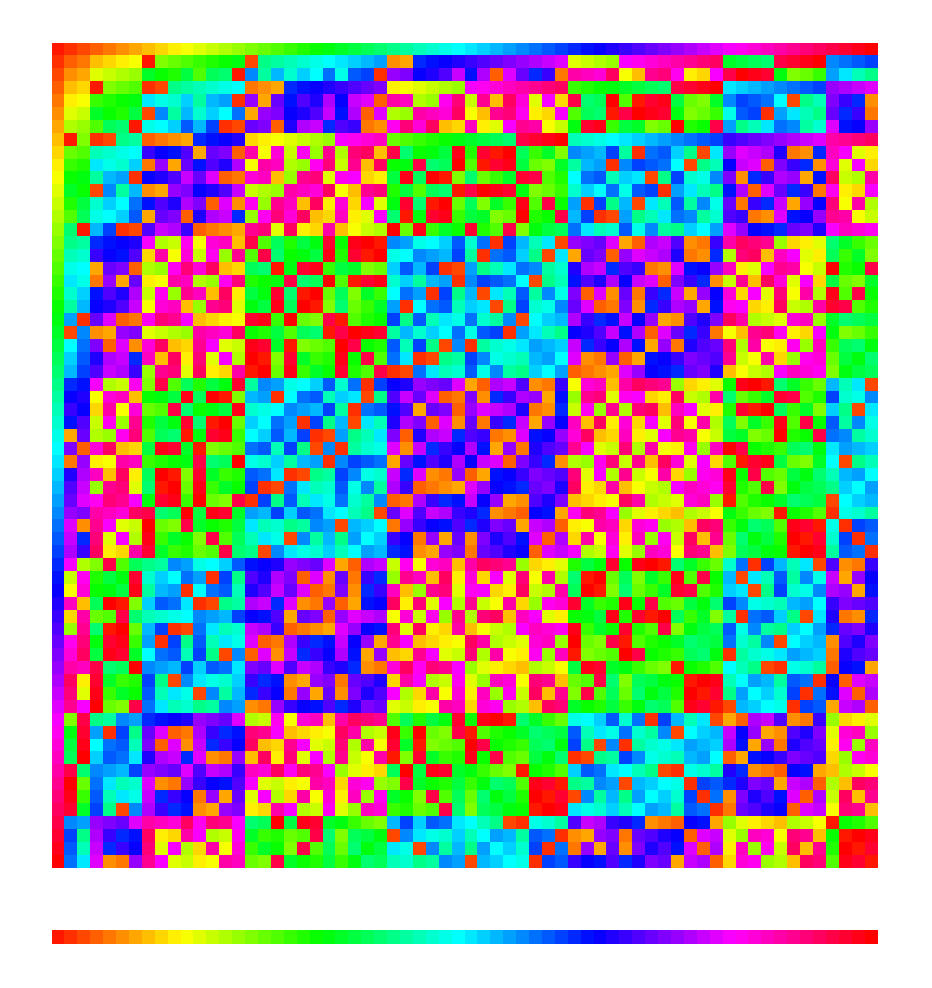}}
  \caption{The multiplication table for the group $\{0,1\}^*/{\sim_{01,q^2+1}}$ with $64$ elements.}
  \label{fig:mult_table}
\end{figure}

\begin{theorem}\label{thm:group_qinvertible}
 Let $u\in A^+$, $\mathfrak{M}\in\mathbb{F}_p[q]$ and the relation $\equiv_{u,\mathfrak{M}}$ defined in \cref{prop:coarsest.cong}. If $q$ is a unit in $\mathbb{K}$, then $A^*/{\equiv_{u,\mathfrak{M}}}$ is a group whose order divides $\per(q)\cdot p^{|u|}$.
\end{theorem}

\begin{proof}
We will prove that $[w]^{\per(q)p^{|u|}}=[\varepsilon]$ for each $w \in A^*$, which suffices for the claim. Note that the length of $w^{\per(q)p^{|u|}}$ is a multiple of $\per(q)$ and thus we trivially have $|w^{\per(q)p^{|u|}}|\equiv 0\pmod{\per(q)}$. Since, for every non-empty factor $v$ of~$u$, $\qbin{\varepsilon}{v}=0$, we have to show that \[\qbin{w^{\per(q)p^{|u|}}}{v}\equiv 0\pmod{\mathfrak{M}}.\]
Let $v$ be a non-empty factor of $u$. We prove by induction on $|v|$ that, for all words $w$ and $j\ge 0$,   
\[\qbin{w^{\per(q)p^{|v|+j}}}{v}\equiv 0 \pmod{\mathfrak{M}}.\]
It is enough to prove it for $j=0$. Indeed, replacing $w$ with $w^p$,
\[(w^p)^{\per(q)p^{|v|}}=w^{\per(q)p^{|v|}.p}=w^{\per(q)p^{|v|+1}}.\]
If $v=a\in A$, we make use of \cref{cor:general}
\[\qbin{w^{\per(q)p}}{a}= \sum_{i=0}^{\per(q)p-1} q^{i|w|} \qbin{w}{a}.\]
Since $q$ is a unit (i.e., $\ind(q)=0$), we deduce that 
\[\sum_{i=0}^{\per(q)p-1} q^{i|w|}
= \sum_{i=0}^{\per(q)-1} \sum_{j=0}^{p-1} q^{(\per(q)j+i)|w|} \equiv \sum_{i=0}^{\per(q)-1} p\, q^{i|w|}\equiv0\pmod{\mathfrak{M}}.\]
Assume that the result holds for words of length up to $k\ge 1$ and consider a word $v$ of length $k+1$. From \cref{cor:general}, we have
\[\qbin{w^{\per(q)p^{k+1}}}{v}=\sum_{\substack{v=v_1\cdots v_p\\ v_i\in A^*}} q^{ \sum_{i=1}^{p-1} |v_i|((p-i)\per(q)p^k|w|-|v_{i+1}\cdots v_p|)} \qbin{w^{\per(q)p^k}}{v_1}\cdots \qbin{w^{\per(q)p^k}}{v_p}\]
and, by induction hypothesis, if one of the $v_i$'s in a factorization of $v$ has a length such that $1\le |v_i|\le k$, then
\[\qbin{w^{\per(q)p^k}}{v_i}\equiv 0\pmod{\mathfrak{M}}.\]
So the only non-zero terms modulo~$\mathfrak{M}$ correspond to factorizations where all the $v_i$'s are $\varepsilon$ except one which is $v$. We get 
\[\qbin{w^{\per(q)p^{k+1}}}{v}\equiv\sum_{j=0}^{p-1} q^{ j|v|\per(q)p^k|w|} \qbin{w^{\per(q)p^k}}{v}\equiv p \qbin{w^{\per(q)p^k}}{v} \equiv 0 \pmod{\mathfrak{M}}.\]
\end{proof}

\begin{corollary}\label{cor:lvrm}
  Let $v$ be a word and $\mathfrak{M}=a(q-1)^d$ for some integer $d\ge 1$ and non-zero $a\in\mathbb{F}_p$.
  The language $L_{v,\mathfrak{R},\mathfrak{M}}$ given in \cref{eq:plang} is a $p$-group language.
\end{corollary}

\begin{proof}
By \cref{prop:order_of_q_power_of_p}, we have $\ind(q) = 0$ and $\per(q) = p^t$ for some $t \geq 0$ (and clearly we must have $p^t \geq d$).
Thanks to \cref{thm:group_qinvertible}, the order of the group $G=A^*/{\equiv_{v,\mathfrak{M}}}$ divides $p^{t+|v|}$, and is thus a $p$-group.
  
Let $\mathfrak{R}\in\mathbb{F}_p[q]$ be a polynomial of degree less than $\deg(M)$.  Consider the projection morphism $\pi:A^*\to G$ that maps a word $x$ to its equivalence class $[x]$. Take $S\subset G$ be the set of classes defined by $[y]\in S$ if and only if \[\qbin{y}{v}\equiv\mathfrak{R}\pmod{\mathfrak{M}}.\]
  It is clear that $\pi^{-1}(S)=L_{v,\mathfrak{R},\mathfrak{M}}$.
\end{proof}

The DFA $\mathcal{M}$ whose set of states is $G=A^*/{\equiv_{v,\mathfrak{M}}}$, with transition function $\delta \colon G\times A\to G$ defined by $\delta([y],a)=[ya]$, initial state $[\varepsilon]$ and set of final states $\pi^{-1}(S)$ as in the previous proof, accepts the language $L_{v,\mathfrak{R},\mathfrak{M}}$. Since $G$ is a group, this DFA is a permutation automaton (i.e., its transition semigroup is a group). A quotient of a permutation automaton being a permutation automaton, the minimal automaton of $L_{v,\mathfrak{R},\mathfrak{M}}$ is a permutation automaton.

\begin{example}
  Continuing \cref{exa:equiv} and applying the procedure described above, the minimization of the DFA is depicted in \cref{fig:min_aut}. For $\mathfrak{R}$ being $0$, $1$, $q+1$ or $q$, the structure of the minimal automaton does not change but one has to choose the four states on the first $\{1,2,3,4\}$, resp.~second $\{5,6,7,8\}$, third $\{9,10,11,12\}$ or fourth layer $\{13,14,15,16\}$ to be the final states.
  \begin{figure}[h!tb]
    \begin{minipage}{.60\linewidth}
  \centering
\begin{tikzpicture}[scale=.6,every node/.style={circle,minimum width=.5cm,inner sep=1pt}]
\node (1) at (0,0) [draw,circle] {\scriptsize{$1$}};
\node (2) at (3,0) [draw,circle] {\scriptsize{$2$}};
\node (3) at (6,0) [draw,circle] {\scriptsize{$3$}};
\node (4) at (3,-3) [draw,circle] {\scriptsize{$6$}};
\node (5) at (9,0) [draw,circle] {\scriptsize{$4$}};
\node (6) at (6,-6) [draw,circle] {\scriptsize{$11$}};
\node (7) at (6,-3) [draw,circle] {\scriptsize{$7$}};
\node (8) at (3,-6) [draw,circle] {\scriptsize{$10$}};
\node (9) at (9,-9) [draw,circle] {\scriptsize{$16$}};
\node (10) at (9,-6) [draw,circle] {\scriptsize{$12$}};
\node (11) at (9,-3) [draw,circle] {\scriptsize{$8$}};
\node (12) at (3,-9) [draw,circle] {\scriptsize{$14$}};
\node (13) at (0,-9) [draw,circle] {\scriptsize{$13$}};
\node (14) at (0,-6) [draw,circle] {\scriptsize{$9$}};
\node (15) at (0,-3) [draw,circle] {\scriptsize{$5$}};
\node (16) at (6,-9) [draw,circle] {\scriptsize{$15$}};
  \draw[->] (1) edge node[below] {\scriptsize{$0$}} (2);
  \draw [->] (1) edge[out=70,in=110,looseness=8] node[above] {\scriptsize{$1$}} (1); 
  \draw[->] (2) edge node[below] {\scriptsize{$0$}} (3); 
  \draw[->] (2) edge node[left] {\scriptsize{$1$}} (4); 
  \draw[->] (3) edge node[below] {\scriptsize{$0$}} (5); 
  \draw[->] (3) edge[out=290,in=70] node[pos=.15, right] {\scriptsize{$1$}} (6); 
  \draw[->] (4) edge node[below] {\scriptsize{$0$}} (7); 
  \draw[->] (4) edge node[left] {\scriptsize{$1$}} (8); 
  \draw[->] (5) edge[out=160,in=20] node[pos=.15, above] {\scriptsize{$0$}} (1); 
  \draw[->] (5) edge[out=290,in=70] node[right] {\scriptsize{$1$}} (9); 
  \draw[->] (6) edge node[below] {\scriptsize{$0$}} (10); 
  \draw[->] (6) edge[out=110,in=250] node[pos=.3, left] {\scriptsize{$1$}} (3); 
  \draw[->] (7) edge node[below] {\scriptsize{$0$}} (11);
  \draw [->] (7) edge[out=-70,in=-110,looseness=8] node[below] {\scriptsize{$1$}} (7); 
  \draw[->] (8) edge node[below] {\scriptsize{$0$}} (6); 
  \draw[->] (8) edge node[left] {\scriptsize{$1$}} (12); 
  \draw[->] (9) edge[out=160,in=20] node[pos=.15, above] {\scriptsize{$0$}} (13); 
  \draw[->] (9) edge node[left] {\scriptsize{$1$}} (10); 
  \draw[->] (10) edge[out=160,in=20] node[pos=.15, above] {\scriptsize{$0$}} (14); 
  \draw[->] (10) edge node[left] {\scriptsize{$1$}} (11); 
  \draw[->] (11) edge[out=160,in=20] node[pos=.15, above] {\scriptsize{$0$}} (15); 
  \draw[->] (11) edge node[left] {\scriptsize{$1$}} (5); 
  \draw[->] (12) edge node[below] {\scriptsize{$0$}} (16); 
  \draw[->] (12) edge[out=70,in=290] node[pos=.15, right] {\scriptsize{$1$}} (2); 
  \draw[->] (13) edge node[below] {\scriptsize{$0$}} (12); 
  \draw[->] (13) edge[out=110,in=250] node[pos=.15, left] {\scriptsize{$1$}} (15); 
  \draw[->] (14) edge node[below] {\scriptsize{$0$}} (8);
  \draw [->] (14) edge[out=70,in=110,looseness=8] node[above] {\scriptsize{$1$}} (14); 
  \draw[->] (15) edge node[below] {\scriptsize{$0$}} (4); 
  \draw[->] (15) edge[out=290,in=70] node[pos=.15, right] {\scriptsize{$1$}} (13); 
  \draw[->] (16) edge node[below] {\scriptsize{$0$}} (9);
  \draw [->] (16) edge[out=-70,in=-110,looseness=8] node[below] {\scriptsize{$1$}} (16);
  \draw [->] (-1,0) -- (1);
\end{tikzpicture}
\end{minipage}
\begin{minipage}{.25\linewidth}
  \[\begin{array}{c|c}
    \mathfrak{R} & \text{final states}\\
    \hline
    0&1,2,3,4\\
    1&5,6,7,8\\
    q+1&9,10,11,12\\
    q&13,14,15,16\\
  \end{array}\]
\end{minipage}
\caption{The structure of the minimal automaton of $L_{01,\mathfrak{R},q^2+1}$.}
  \label{fig:min_aut}
\end{figure}
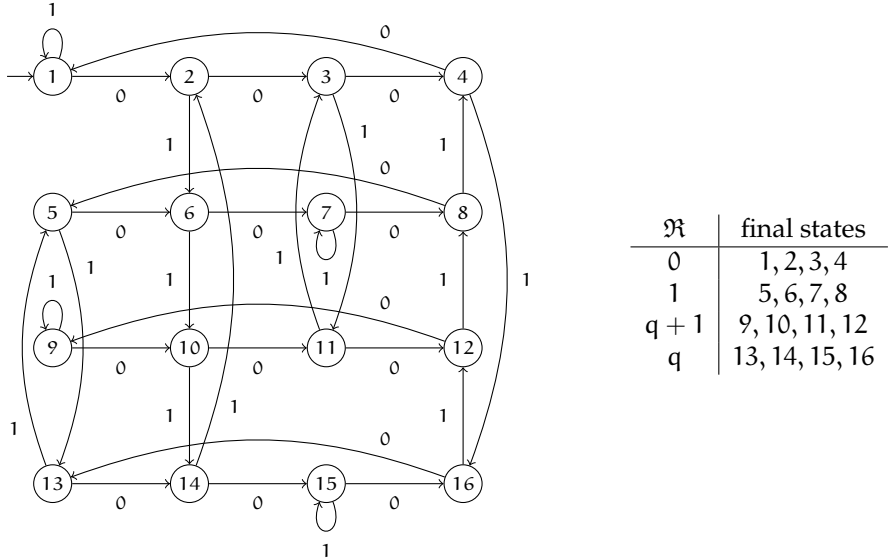
In \cref{fig:min_aut}, we see how $0$ and $1$ act as permutations of the set of states. These two permutations may be used as generators of the considered group.
\end{example}

We have a direct generalization of Eilenberg's theorem characterizing $p$-group languages.
\begin{theorem}
  Let $p$ be a prime and $\mathfrak{M} = a(q-1)^d$ with $d\ge 1$ an integer and a non-zero $a\in\mathbb{F}_p$. A language is a $p$-group language if and only if it is
  a Boolean combination of languages of the form
  \[
  L_{v,\mathfrak{R},\mathfrak{M}}=\left\{u\in A^*\mid \qbin{u}{v}\equiv \mathfrak{R}\pmod{\mathfrak{M}}\right\}.
  \]
\end{theorem}

\begin{proof}
Since $\mathfrak{M}$ is of the form $a(q-1)^d$, if $\mathfrak{P}\equiv\mathfrak{R}\pmod{\mathfrak{M}}$, then $\mathfrak{P}(1)=\mathfrak{R}(1)$. Since a classical binomial coefficient $\binom{u}{v}$ of words  is a $q$-binomial $\binom{u}{v}_1$ evaluated at $1$, we have $u\in L_{v,r,p}$ if and only if $u\in L_{v,\mathfrak{R},\mathfrak{M}}$ for some $\mathfrak{R}$ such that $\mathfrak{R}(1)\equiv r\pmod{p}$. 

This means that every language $L_{v,r,p}:=\{u\in A^*\mid \binom{u}{v}\equiv r\pmod{p}\}$ is the disjoint union of languages  \cref{eq:plang} of the form $L_{v,\mathfrak{R},\mathfrak{M}}$ for polynomials $\mathfrak{R}$ such that $\mathfrak{R}(1)=r$ modulo~$p$. 
More precisely, we observe that amongst the $p^d$ polynomials $\mathfrak{R}$ of degree less than $d$ in $\mathbb{F}_p[q]$, exactly $p^{d-1}$ are such that $\mathfrak{R}(1)$ is a given element in $\mathbb{F}_p$; the sum modulo~$p$ of the elements in a $d$-tuple of coefficients in $\mathbb{F}_p$ is completely determined by one of the coefficients when the other $d-1$ ones are fixed.

 We may now conclude using Eilenberg's theorem. If $L$ is a $p$-group language, it  is a Boolean combination of languages of the form $L_{v,r,p}$ which are themselves finite unions of languages of the form $L_{v,\mathfrak{R},\mathfrak{M}}$.

  The converse comes from \cref{cor:lvrm}. Recall that any Boolean combination of $p$-group languages is a $p$-group language.
\end{proof}

 As an example, for $p=2$ and $t=1$, since $q+1$ at $q=1$ evaluates to $0$ modulo~$2$
\[L_{v,0,2}=\left\{u\in A^*\mid \qbin{u}{v}\equiv 0\pmod{q^2-1}\right\}\cup
\left\{u\in A^*\mid \qbin{u}{v}\equiv q+1\pmod{q^2-1}\right\}\]
and since $q$ at $q=1$ evaluates to $1$ modulo~$2$
\[L_{v,1,2}=\left\{u\in A^*\mid \qbin{u}{v}\equiv 1\pmod{q^2-1}\right\}\cup
\left\{u\in A^*\mid \qbin{u}{v}\equiv q\pmod{q^2-1}\right\}.\]

In \cref{fig:min_aut2}, we have depicted the minimal automaton for $L_{01,r,2}$. For $r=0$ (resp. $r=1$), the final states are $\{1,2\}$ (resp. $\{3,4\}$). This automaton is a quotient of the one depicted in \cref{fig:min_aut} where the states $1,3,9,11$ are mapped to $1$, resp. $2,4,10,12$ to $2$, resp. $6,8,14,16$ to $3$ and $5,7,13,15$ to $4$. 
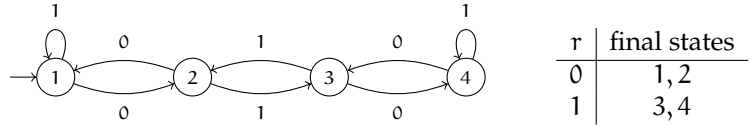
\begin{figure}[h!tb]
  \centering
  \begin{minipage}{.5\linewidth}
 \centering
    \begin{tikzpicture}[scale=.6,every node/.style={circle,minimum width=.5cm,inner sep=1pt}]
\node (1) at (0,0) [draw,circle] {\scriptsize{$1$}};
\node (2) at (3,0) [draw,circle] {\scriptsize{$2$}};
\node (3) at (6,0) [draw,circle] {\scriptsize{$3$}};
\node (5) at (9,0) [draw,circle] {\scriptsize{$4$}};
  \draw[->] (1) edge[out=-20,in=-160] node[below] {\scriptsize{$0$}} (2);
  \draw [->] (1) edge[out=70,in=110,looseness=8] node[above] {\scriptsize{$1$}} (1); 
  \draw[->] (2) edge[out=-20,in=-160] node[below] {\scriptsize{$1$}} (3); 
  \draw[->] (2) edge[out=160,in=20] node[above] {\scriptsize{$0$}} (1); 
  \draw[->] (3) edge[out=160,in=20] node[above] {\scriptsize{$1$}} (2); 
  \draw[->] (3) edge[out=-20,in=-160] node[below] {\scriptsize{$0$}} (5);
  \draw[->] (5) edge[out=160,in=20] node[above] {\scriptsize{$0$}} (3);
  \draw [->] (5) edge[out=70,in=110,looseness=8] node[above] {\scriptsize{$1$}} (5); 
  \draw [->] (-1,0) -- (1);
\end{tikzpicture}
\end{minipage}
\begin{minipage}{.2\linewidth}
  \[\begin{array}{c|c}
    r & \text{final states}\\
    \hline
    0&1,2\\
    1&3,4\\
  \end{array}\]
\end{minipage}
\caption{The structure of the minimal automaton of $L_{01,r,2}$.}
  \label{fig:min_aut2}
\end{figure}

\section{Conclusions}
In this paper, we have proposed a $q$-deformation of the binomial coefficients of words. Our natural definition seems quite reasonable, since it has enabled us to obtain the analogue of many classical properties or formulas. Moreover, as shown by \cref{cor:reconstruction}, these $q$-deformations of binomial coefficients of words contain much richer information than the coefficients obtained for $q=1$. However, it should be noted that not all classical properties extend in an obvious way. For example, our proposed family of $q$-deformations of the infiltration product does not yield an analogue to Chen--Fox--Lyndon relation \cref{eq:cfl}. Finally, it is interesting to note that as an application of our developments, we obtain a refinement of Eilenberg's theorem characterizing $p$-group languages. One also obtains other group languages. It could be interesting to investigate their properties further.

\section*{Acknowledgments}
We thank S.~Morier-Genoud and C.~Reutenauer for useful references. We warmly thank the anonymous referees. They have made a significant contribution to improving the presentation of this text and have proposed changes to simplify a number of proofs. In particular, \cref{cor:ref} and \cref{rem:ref} were suggested by one of the referees.

\section*{Data availability}

Data sharing is not applicable to this article as no datasets were generated or analyzed.

\section*{Author Contributions}

Each of the three authors have equally contributed to the manuscript. 

\section*{Conflict of interest statement}

The authors declare no conflict of interest.

\bibliographystyle{plainurl}
\bibliography{./bibliography}

\end{document}